\documentclass[11pt,reqno]{amsart}
\usepackage{fullpage}
\usepackage[margin=1in]{geometry}
\usepackage{verbatim}
\usepackage{lineno}
\usepackage{amsmath}
\usepackage{bbm}
\usepackage{graphicx}
\usepackage{float}
\usepackage{mathrsfs}
\usepackage[english]{babel}
\usepackage[utf8]{inputenc}
\usepackage{amsthm}
\usepackage{amssymb}
\usepackage{listings}
\usepackage{xcolor}
\usepackage{enumerate}
\usepackage{thmtools}
\usepackage{thm-restate}
\usepackage[ruled,vlined]{algorithm2e}
\usepackage{epsfig}
\providecommand{\U}[1]{\protect\rule{.1in}{.1in}}
\usepackage[all]{xy}
\usepackage{caption}
\usepackage{longtable}
\usepackage{mathtools}
\usepackage{tikz-cd}
\usepackage{subcaption}
\usepackage{url}
\usepackage[pdftex,pdfborder={0 0 0},
colorlinks=true,
linkcolor=blue,
citecolor=red,
pagebackref=false,
]{hyperref}
\usepackage{tcolorbox}
\newcommand{\R}{\mathbb{R}}

\newcommand{\bga}{\bar{\gamma}}
\newcommand{\beps}{\bar{\epsilon}}
\newcommand{\bdel}{\bar{\delta}}

\def\hjb{\textbf{HJB}}
\def\fp{\textbf{FP}}
\def\v{\textbf{V}}

\def\br{\textbf{BR}}
\def\hjbt{\textbf{HJB}_{\theta}}
\def\vt{\textbf{V}_{\theta}}
\def\fpt{\textbf{FP}_{\theta}}
\def\brt{\textbf{BR}_{\theta}}

\DeclareMathOperator*{\argmin}{argmin}

\newtheorem{thm}{Theorem}[section]
\newtheorem{prop}[thm]{Proposition}
\newtheorem{lem}[thm]{Lemma}
\newtheorem{cor}[thm]{Corollary}

\theoremstyle{definition}
\newtheorem{ass}{Assumption}

\newtheorem{ass2}{Assumption}
\newtheorem{defn}[thm]{Definition}

\theoremstyle{remark}
\newtheorem{rem}[thm]{Remark}

\numberwithin{equation}{section}

\definecolor{dgreen}{rgb}{0.00,0.49,0.00}
\definecolor{Brown}{rgb}{0.45,0.0,0.05}

\title{A mesh-independent method for second-order potential mean field games}

\date{\today}

\author{Kang Liu$^{1,2}$ and Laurent Pfeiffer$^1$}
\address{$^1$Universit{\'e} Paris-Saclay, CNRS, CentraleSup{\'e}lec, Inria, Laboratoire des signaux et syst{\`e}mes, 91190, Gif-sur-Yvette, France.}
\address{$^2$Institut Polytechnique de Paris, CNRS, Ecole Polytechnique, CMAP, 91120 Palaiseau, France.}
\email{kang.liu@polytechnique.edu, laurent.pfeiffer@inria.fr}

\begin{document}

\begin{abstract}
This article investigates the convergence of the Generalized Frank-Wolfe (GFW) algorithm for the resolution of potential and convex second-order mean field games. More specifically, the impact of the discretization of the mean-field-game system on the effectiveness of the GFW algorithm is analyzed. The article focuses on the theta-scheme introduced by the authors in a previous study. A sublinear and a linear rate of convergence are obtained, for two different choices of stepsizes. These rates have the mesh-independence property: the underlying convergence constants are independent of the discretization parameters.
\end{abstract}

\maketitle

\section{Introduction}

\subsection{Context and main contributions}

This article is concerned with the numerical resolution of second-order mean field games (MFGs). These models describe the asymptotic behavior of Nash equilibria in stochastic differential games, as the number of players goes to infinity. They were introduced independently in \cite{lasry2007mean} and \cite{huang2006large} and have applications in various domains, such as economics, biology, finance and social networks, see \cite {djehiche2017mean}. 
In this article, we consider the following standard second-order MFG on the space $Q \coloneqq [0,1]\times \mathbb{T}^d$,
\begin{equation}\label{eq:mfg}\tag{MFG}
\left\{\begin{array}{lll}
\text{(i)} & -\partial_{t} u-\sigma \Delta u+H^c\left(t, x, \nabla u(t, x)\right) = f^c(t,x , m(t)) & (t,x) \in Q, \\
\text{(ii)} & v( t,x )=-H^c_{p}\left( t, x, \nabla u(t, x)\right) & (t, x) \in Q, \\
\text{(iii)} & \partial_{t} m-\sigma \Delta m+\operatorname{div}(v m)=0 & (t, x) \in Q, \\
\text{(iv)} \ & m( 0,x)=m^c_{0}(x), \quad u(1,x)=g^c(x) & x \in \mathbb{T}^d,
\end{array}\right.
\end{equation}
where the Hamiltonian $H^c$ is related to the Fenchel conjugate of a running cost $\ell^c \colon Q\times \mathbb{R}^d \rightarrow \mathbb{R}$:
\begin{equation*}
H^c(t,x,p) \coloneqq \sup_{v\in \mathbb{R}^d} \ \langle -p, v\rangle - \ell^c(t,x,v).
\end{equation*}
The existence and uniqueness of the classical solution of \eqref{eq:mfg} is proved in \cite{lasry2007mean} under assumptions on the coupling function $f^c$.

Several discretization schemes have been proposed and analyzed for the resolution of \eqref{eq:mfg}. They consist of a backward discrete Hamilton-Jacobi-Bellman (HJB) equation and a forward discrete Fokker-Planck (FP) equation: They preserve the nature of the problem as a coupled system of two equations.
An implicit finite difference scheme has been introduced in \cite{achdou2010mean} and convergence results for this scheme have been obtained in \cite{achdou2013mean} and \cite{achdou2016}. 
Other schemes, based on semi-Lagrangian discretizations have been investigated in \cite{carlini2014,carlini2015,hadikhanloo2019finite} for first-order and (possibly) degenerate second-order MFGs.
This work will focus on a scheme called theta-scheme, recently introduced by the authors in \cite{BLP2022}. In short, this scheme involves a Crank-Nicolson discretization of the diffusion term and an explicit discretization of the first-order non-linear term.

The resolution of the discretized coupled system is in general a difficult task. We restrict our attention to the case of potential MFGs, for which optimization methods can be leveraged.
The system \eqref{eq:mfg} is said to be potential (or variational) if there exists a function $F^c \colon [0,1]\times \mathcal{D} \to \mathbb{R}$ such that for any $t\in [0,1]$ and $m_1,m_2 \in \mathcal{D}$ (see the definition of $\mathcal{D}$ in \eqref{eq:def_D}),
\begin{equation}\label{eq:potential}
    F^c(t,m_1) - F^c(t,m_2) = \int_{0}^1 \int_{x\in \mathbb{T}^d} f^c(t,x, m_1 + s(m_2-m_1)) (m_2(x) - m_1(x))dx ds.
\end{equation}
We further restrict our attention to the case of a convex potential function $F^c$.
In the presence of such $F^c$, the system \eqref{eq:mfg} can be interpreted as the first-order optimality condition of an optimal control problem driven by the FP equation,
\begin{equation}\label{pb:OC}\tag{OC}
\begin{cases}
\ \underset{(m,v)}{\inf}  \
J^c(m,v) \coloneqq {\displaystyle \int_{Q} }\ell^c(t,x,v) m(t,x) dt dx + {\displaystyle \int_{0}^1} F^c(t,m(t)) dt+ {\displaystyle \int_{\mathbb{T}^d} } g^c(x)m(T,x) dx, \\[1.2em]
\ \text{such that } \begin{cases}
\ \partial_{t} m-\sigma \Delta m+\operatorname{div}(v m)=0 ,   \quad & \forall (t, x) \in Q, \\
\ m( 0,x)=m^c_{0}(x), &\forall x\in \mathbb{T}^d.
\end{cases} 
\end{cases}
\end{equation} 
Problem \eqref{pb:OC} is equivalent to a convex optimal control problem, obtained through the classical Benamou-Brenier transform \cite{benamou2017}. 
Then, some numerical algorithms can be applied to find a solution of \eqref{pb:OC}, such as ADMM \cite{benamou2015augmented,andreev2017}, the Chambolle-Pock algorithm \cite{achdoumean}, the fictitious play \cite{hadikhanloo2019finite} and the generalized Frank-Wolfe (GFW) algorithm \cite{lavigne2021}.
Some articles propose to discretize the optimal control problem \eqref{pb:OC}, see for example \cite{lachapelle2010computation,andreev2017}.
In this context, it is very desirable that the potential structure of the continuous MFG is preserved at the level of the discretized coupled system, so that one can apply in a direct fashion suitable optimization methods to the discrete system.
This is in particular the case for the implicit scheme proposed in \cite{achdou2013mean} and solved in 
\cite{achdoumean} with the Chambolle-Pock algorithm.
As we establish in this article, the theta-scheme of \cite{BLP2022} also preserves the potential structure of the MFG system.

We focus in this article on the resolution of the discrete MFG system with the Generalized Frank-Wolfe (GFW) algorithm see \cite{bredies2009}. This algorithm is an iterative method, consisting in solving at each iteration a partially linearized version of the potential problem \eqref{pb:OC}.
The linearized problem to be solved is equivalent to a stochastic optimal control problem that can be solved by dynamic programming. As we will explain more in detail, this allows to interpret the GFW method as a best-response procedure. For a specific choice of stepsize, it coincides with the fictitious play method of \cite{cardaliaguet2017learning}. Others works have investigated the fictitious play method for MFGs: \cite{perrin2020fictitious} proves the convergence of the continuous method, in a discrete setting with common noise; \cite{hadikhanloo2019finite} proves a general result for fully discrete MFGs, which can be applied to discretized first-order MFGs. The article \cite{geist2022concave} shows the connexion between fictitious play and the Frank-Wolfe algorithm for potential discrete MFGs.

The general objective of the article is to show that the performance of the GFW algorithm is not impacted by a refinement of the discretization grid.
The main results of our article are two mesh-independence properties for the resolution of \eqref{eq:mfg} with the theta-scheme and the GFW algorithm. The terminology mesh-independence was coined in the article \cite{allgower1986}. It is said that an algorithm satisfy a mesh-independence property when approximately the same number of iterations is required to satisfy a stopping criterion, when comparing an infinite-dimensional problem and its discrete counterpart. In a more precise fashion, we will say that the GFW algorithm has a mesh-independent sublinear rate of convergence if there exists a constant $C>0$, independent of the discretization parameters, such that
\begin{equation*}
J_h(m_h^k,v_h^k) - J_h^* \leq \frac{C}{k}, \quad \forall k \geq 1.
\end{equation*}
In the above estimate, $J_h$ denotes the discretized counterpart of $J$, $J_h^*$ denotes the value of the discretized optimal control problem, and $(m_h^k,v_h^k)$ denotes the candidate to optimality obtained at iteration $k$.
Similarly, we will say that the GFW algorithm has a mesh-independent linear rate of convergence if there exist two constants $C>0$ and $\delta \in (0,1)$ such that
\begin{equation*}
J_h(m_h^k,v_h^k) - J_h^* \leq  C \delta^k, \quad \forall k \geq 1.
\end{equation*}

We establish that the GFW algorithm has a mesh-independent sublinear (resp.\@ a linear) rate for two different choices of stepsize. Our analysis is close to the one performed in \cite{lavigne2021}, in which the sublinear and the linear convergence of the GFW algorithm is demonstrated for the continuous model and for the same choices of stepsizes as in the present study. While the sublinear convergence of the GFW method (in a general setting) is classical, the linear rate of convergence relies on recent techniques from \cite{kunisch2021}.
To the best of our knowledge, in the context of mean field games, the mesh-independence property has never been established so far for any other method. Though it seems a natural property, it may not hold in general. In particular, it might not hold for primal-dual methods, whose application relies on a saddle-point formulation of the convex counterpart of \eqref{pb:OC} of the form, in which the Fokker-Planck equation is ``dualized". This saddle-point formulation involves a linear operator, encoding the (discrete) Fokker-Planck equation (see for example \cite[Sec.\@ 3.2]{achdoumean}). As the discretization parameters decrease, the operator norm of these operators (for the Euclidean norm) increases, which has an impact on the convergence properties of methods such as the Chambolle-Pock algorithm. In contrast, the discrete Fokker-Planck equation remains satisfied at each iteration of the GFW equation.

This article is organized as follows. In Section \ref{sec:dmfg}, we introduce some preliminary results and notations. In Section \ref{sec:FW}, we introduce a general class of potential discrete MFGs, containing the theta-scheme. We establish the sublinear and the linear convergence of the GFW method in this discrete setting. We give explicit formulas for the convergence constants. These constants essentially depend on the Lipschitz-modulus of the  coupling function of the MFG and on two bounds, for different norms, of the solution of the discretized Fokker-Planck equation, denoted $C_1$ and $C_2$. We recall the theta-scheme in Section \ref{sec:mesh-indep}, we show that it preserves the potential structure of the MFG, and we prove that the GFW algorithm has a mesh-independent sublinear and linear rates of convergence. The technical analysis relies on precise estimates of the constants $C_1$ and $C_2$, obtained thanks to a general energy estimate and an $L^\infty$-estimate for the discrete Fokker-Planck equation.

\subsection{Notation}

We discretize the interval $[0,1]$ with a time step $\Delta t= 1/T$, where $T\in \mathbb{N}_{+}$. 
The time set is denoted by $\mathcal{T}$ ($\tilde{\mathcal{T}}$ if the final time step $T$ is included). 
Given a finite subset $S$ of $\R^d$, we denote by  $\mathbb{R}(S)$ (resp.\@ $\mathbb{R}^d(S)$) the set of functions from $S$ to $\R$ (resp.\@ $\R^d$). We also denote by $\mathcal{P}(S)$ the of probability measures over $S$.
We call curve of probability measures any function $m\colon  \mathcal{\tilde{T}}\times S \rightarrow \mathbb{R}$ such that $m(t,\cdot) \in \mathcal{P}(S)$, for any $t \in \mathcal{T}$. The set of probability curves is denoted by $\mathcal{P}(\tilde{\mathcal{T}},S)$. In mathematical terms,
\begin{align*}
& \mathcal{T} =\{0,1,\ldots,T-1\}, \qquad  \mathcal{\tilde{T}} = \{0,1,\ldots, T\}; \\[0.5em]
& \mathbb{R}(S) = \{ m \colon S \rightarrow \mathbb{R} \}, \qquad \mathbb{R}^d(S) = \{ m \colon S \rightarrow \mathbb{R}^d \}; \\[0.5em]
&\mathcal{P}(S) = \Big\{m\in \mathbb{R}(S) \, \mid \, \forall x\in S, \, m(x)\geq 0, \, \,\sum_{y\in S} m(y) =1 \Big\}; \\
& \mathcal{P}(\tilde{\mathcal{T}},S)  = \left\{  m \in \mathbb{R}( \tilde{\mathcal{T}}\times S) \, \mid     \, \forall t \in \tilde{\mathcal{T}}, \, m(t,\cdot) \in \mathcal{P}(S)  \right \}.
\end{align*}
We denote by $\|\cdot\|$ and $\langle \cdot ,\cdot
\rangle$ the Euclidean norm and the scalar product in $\mathbb{R}^d$.
Let $S_1$ and $S_2$ be two finite sets. Let $\mu\in \mathbb{R}^d(S_1\times S_2)$. For any $x \in S_1$ and for any $p_1$ and $p_2 \in [1,\infty]$, we denote by $\|\mu(x,\cdot)\|_{p_2}$ the $L^{p_2}$-norm of the function $y \mapsto \mu(x,y)$, defined as follows:
\begin{equation*}
       \|\mu(x,\cdot)\|_{p_2}  = \begin{cases}
       \, \Big(\sum_{y\in S_2} \|\mu(x,y)\|^{p_2} \Big)^{1/{p_2}} ,\quad &\text{if } p_2 \in [1,\infty),
       \\
        \, \max_{y\in S_2} \|\mu(x,y)\|, \quad &\text{if } p_2 = \infty.
       \end{cases}
\end{equation*}
We next define
\begin{equation*}
\|\mu\|_{p_1,p_2} =  \begin{cases}
\, \Big(\sum_{x\in S_1} \|\mu(x,\cdot)\|_{p_2}^{p_1} \Big)^{1/p_1}, \quad &\text{if } p_1\in [1,\infty),\\
\, \max_{x\in S_1} \|\mu(x,\cdot)\|_{p_2}, & \text{if } p_1=\infty.
\end{cases}
\end{equation*}

\begin{lem}[H\"{o}lder's inequality]
Let $S_1$ and $S_2$ be two finite sets.
Let $\mu$ and $\nu  \in \mathbb{R}^n(S_1\times S_2)$ and let $p_1$ and $p_2 \in [1,\infty]$. It holds that
\begin{equation*}
\sum_{x_1\in S_1} \sum_{x_2\in S_2}\Big|\langle \mu(x_1,x_2), \nu(x_1,x_2) \rangle \Big| \leq \|\mu\|_{p_1,p_2} \|\nu\|_{p_1^{*},p_2^{*}},
\end{equation*}
where $1/p_i + 1/ p_i^{*} = 1$, for $i=1,2$.
\end{lem}

\begin{defn}[Nemytskii operators]
Let $\zeta$ be a function from $\mathcal{X}\times \mathcal{Y}$ to $\mathcal{Z}$ and let $u$ be a function from $\mathcal{X}$ to $\mathcal{Y}$. Then the Nemytskii operator is the mapping $\zeta[u]$ from $\mathcal{X}$ to $\mathcal{Z}$ defined by
\begin{equation*}
\zeta[u](x) = \zeta(x,u(x)), \quad \forall x \in \mathcal{X}.
\end{equation*}
\end{defn}

\section{Potential discrete mean field games} \label{sec:dmfg}

In this section we introduce a general class of discrete MFGs containing the $\theta$-scheme of \cite{BLP2022}. We provide a first potential formulation of the discrete MFGs and show their equivalence with a convex optimization problem, using the classical Benamou-Brenier transformation. The analysis being rather standard (see for example \cite{bonnans2023discrete}), we mostly give succinct proofs.

\subsection{Problem formulation}\label{subsec:DMFG_formulation}

We fix $T \in \mathbb{N}_+$ and a finite subset $S$ of $\R^d$.
For the description of the MFG model, we fix a running cost $\ell$, a coupling cost $f$, an initial condition $m_0$, and a terminal cost $g$, where
\begin{align*}
\ell \colon \mathcal{T}\times S \times \mathbb{R}^d \rightarrow \mathbb{R} \cup \{ \infty \},
\qquad
f \colon \mathcal{T}\times S \times \mathbb{R}(S) \rightarrow \mathbb{R},
\qquad
m_0 \in \mathcal{P}(S),
\qquad
g \in \mathbb{R}(S).
\end{align*}
For a given $m \in \mathcal{P}(\tilde{\mathcal{T}}, S)$, we denote by $\bar{\ell}_m$ the map defined by
\begin{equation*}
\bar{\ell}_{m} \colon (t,x,\omega) \in \mathcal{T} \times S \mapsto
\ell(t, x, \omega) + f(t, x,m(t)),
\end{equation*}
where $m(t)= (m(t,x))_{x \in S}$.
We require the following assumptions for $\ell$, $m$, $f$, and $m_0$.

\begin{ass2} \label{ass1}
\begin{enumerate}
\item \emph{Bounded domain.} 
For all $(t,x) \in \mathcal{T} \times S$, $\ell(t,x,\cdot)$ is lower semi-continuous with a non-empty domain. There exists $D>0$ such that for any $t\in \mathcal{T}$, for any $x\in S$, and for any $\|v\| > D$, we have $\ell(t,x,v) = +\infty$. 
\item \emph{Regularity.} 
There exists $L_{f}$ such that for any $(t,x)$ and for any $m_1$, and $m_2$ in $\mathcal{P}(S)$, we have
\begin{equation*}
    |f(t,x,m_1) - f(t,x,m_2)| \leq L_f \|m_1 - m_2\|_2.
\end{equation*}
\item \emph{Strong convexity.}
There exists $\alpha > 0$ such that for any $t \in \mathcal{T}$ and for any $x \in S$, the function $\ell(t,x,\cdot)$ is $\alpha$-strongly convex, i.e.,
\begin{equation*}
\ell(t,x,v_2) \geq \ell(t,x,v_1) + \langle p, v_2-v_1\rangle + \frac{\alpha}{2}\|v_2-v_1\|^2,
\end{equation*}
    for all $v_1$ and $v_2$ in $\R^n$ and for all $p \in  \partial_p \ell(t,x,v_1)$.
\item \emph{Monotonicity.} For any $t \in \mathcal{T}$, for any $m_1$ and $m_2$ in $\mathcal{P}(S)$,
\begin{equation*}
\sum_{x\in S} \big( f(t,x,m_1)-f(t,x,m_2) \big) (m_1(x)-m_2(x) ) \geq 0.
\end{equation*}
\end{enumerate}
\end{ass2}

We now fix two elements $\pi_0 \in \mathbb{R}(\mathcal{T}\times S^2)$ and $\pi_1 \in \mathbb{R}^d(\mathcal{T}\times S^2)$ and define the map $\pi$ by
\begin{equation*}
\pi \colon (t,x,y,\omega) \in \mathcal{T} \times S \times S \times \R^d
\mapsto \pi_0(t,x,y) + \Delta t \langle \pi_1(t,x,y), \omega \rangle.
\end{equation*}
The map $\pi$ describes the probability of an agent located at time $t$ in state $x$, using the control $\omega$, to reach state $y$ at time $t+1$. Our analysis will exploit the fact that $\pi$ is affine with respect to $\omega$. Recalling the constant $D$ introduced in Assumption \ref{ass1}, we consider the following assumption.

\begin{ass2}\label{ass2}
The elements $\pi_0$ and $\pi_1$ satisfy the following:
\begin{equation}\label{cond1}
\begin{cases}
\begin{array}{ll}
      \pi_0(t,x,\cdot) \in \mathcal{P}(S),  \quad  &\forall (t,x) \in \mathcal{T}\times S, 
      \\[0.5em]
      \sum_{y\in S}\pi_1 (t,x,y) = 0 , \quad  &\forall (t,x) \in \mathcal{T}\times S, \\[0.5em]
      \pi_0(t,x,y) \geq  \Delta t D \| \pi_1(t,x,y)\| , \qquad &\forall (t,x,y) \in \mathcal{T}\times S\times S.
\end{array}
\end{cases}
\end{equation}
\end{ass2}

An immediate consequence of Assumption \ref{ass2} is the following: For all $(t,x) \in \mathcal{T} \times S$, for all $\omega \in \R^d$, if $\| \omega \| \leq D$, then $\pi(t,x,\cdot,\omega) \in \mathcal{P}(S)$.

\begin{ass2} \label{ass3}
There exists a function $F \colon \mathcal{T}\times \mathcal{P}(S) \to \mathbb{R}^d$ such that for any $t\in \mathcal{T}$ and for any  $m_1$ and $m_2$ in $\mathcal{P}(S)$, it holds
\begin{equation*}
    F(t,m_1) - F(t,m_2) = \int_{0}^1\sum_{x\in S} f(t,x, m_1 + s(m_2-m_1)) (m_2(x) - m_1(x)) ds.
\end{equation*}
\end{ass2}

We have the following convexity property for the potential function $F$.
 
\begin{lem} \label{lm:F_bis}
For any $t\in \mathcal{T}$ and for any $m_1$ and $m_2$ in $\mathcal{P}(S)$, it holds that
\begin{align}
     F(t,m_2) & \geq F(t,m_1) + \sum_{x\in S} f(t,x, m_1 ) (m_2(x) - m_1(x)). \label{eq:linear}
\end{align}
\end{lem}

\begin{proof}
Inequality \eqref{eq:linear} follows from the definition of $F$ and the monotonicity in Assumption \ref{ass1}.
\end{proof}

Until the end of Section \ref{sec:FW}, we assume that Assumptions \ref{ass1}, \ref{ass2}, and \ref{ass3} are satisfied. Following \cite[Sec.\@ 2.3]{BLP2022}, we consider the following discrete mean field game system, involving the variables $u \in \mathbb{R}(\tilde{\mathcal{T}} \times S)$, $v \in \mathbb{R}^d(\mathcal{T}\times S)$, and $m \in \mathcal{P}(\tilde{\mathcal{T}}, S)$:
\begin{equation}\label{eq:dmfg}\tag{DMFG}
\left\{\begin{array}{cll}
\mathrm{(i)} \ & u =  \hjb (m) ,\\[0.3em]
\mathrm{(ii)} \  & v = \v (u),\\[0.3em]
\mathrm{(iii)} \ & m = \fp (v),
\end{array}\right.
\end{equation}
where the Hamilton-Jacobi-Bellman mapping $\hjb$, the optimal control mapping $\v$, and the Fokker-Planck mapping $\fp$ are defined as follows:
\begin{itemize}
\item Given $m \in \mathcal{P}(\tilde{\mathcal{T}},S)$, $u= \hjb(m) \in \R(\tilde{\mathcal{T}}\times S)$ is the solution to
\begin{equation}\label{eq:hjbmap}
\begin{cases}
\     u(t,x) = \underset{\omega\in\mathbb{R}^d}{\inf} \ \bar{\ell}_{m}(t,x,\omega) \Delta t + \sum_{y\in S} \pi(t,x,y,\omega )u(t+1, y) ,\quad &\forall(t,x)\in \mathcal{T}\times S, \\[0.8em]
\      u(T,x) = g(x) , &\forall x\in S.
\end{cases} 
\end{equation}
\item Given $u \in \R(\mathcal{T} \times S)$, $v=\v(u) \in \mathbb{R}^d(\mathcal{T}\times S)$ is defined by
\begin{equation}\label{eq:vmap}
     v(t,x) = \argmin_{\omega \in \mathbb{R}^d}\ \ell(t,x,\omega)\Delta t + \sum_{y\in S} \pi(t,x,y,\omega )u(t+1, y) , \quad \forall(t,x)\in \mathcal{T}\times S.
\end{equation}
\item Given $v \in  \mathbb{R}^d(\mathcal{T}\times S) $, $m= \fp(v) \in \mathbb{R}(\tilde{\mathcal{T}}\times S)$ is defined as the solution to
 \begin{equation}
    \begin{cases}\label{eq:fpmap}
\       m(t+1,y) = \sum_{x\in S} \pi(t,x,y,v(t,x)) m(t,x),  & \forall (t,y) \in \mathcal{T}\times S, \\[0.8em]
\       m(0,x)  = m_0(x), & \forall x \in S.
\end{cases}
\end{equation}
\end{itemize}

\begin{lem}[Continuity of \textnormal{\hjb}] \label{lm:hjb_stab}
For any $m_1$ and $m_2$ in $\mathcal{P}(\tilde{\mathcal{T}},S)$, we have
\textnormal{
\begin{equation}\label{eq:hjbstable}
    \|\hjb(m_1) - \hjb(m_2)\|_{\infty,\infty} \leq L_f \|m_1 - m_2\|_{\infty,2}.
\end{equation}}
\end{lem}

\begin{proof}
    See \cite[Eq.\@ A.6]{BLP2022}.
\end{proof}

\begin{lem}[Continuity of \textnormal{\fp}] \label{lemma:cont_FP}
Let $v_1$ and $v_2$ in $\R^d(\mathcal{T} \times S)$ be such that $\| v_1 \|_{\infty,\infty} \leq D$ and $\| v_2 \|_{\infty,\infty} \leq D$. There exists a constant $C$, independent of $v_1$ and $v_2$, such that
\begin{equation} \label{eq:energy2_init}
    \big\| \textnormal{\fp}(v_1) - \textnormal{\fp}(v_2) \big\|^2_{\infty,2} \leq C \Delta t \sum_{t\in \mathcal{T}}\sum_{x\in S} \|(v_1 - v_2)m_1(t,x)\|^2. 
\end{equation}
As an immediate consequence, the mapping $\textnormal{\fp}$ is continuous.
\end{lem}

\begin{proof}
Let $m_1= \fp(v_1)$ and $m_2= \fp(v_2)$.
Let $\mu= m_1-m_2$. It is easy to verify that
\begin{equation}\label{eq:delta_m}
 \begin{cases}
 \      \mu(t+1,y) = \sum_{x\in S} \pi(t,x,y,v_1(t,x)) \mu(t,x) +  \Delta t \sum_{x\in S} \langle \pi_1(t,x,y), (v_2-v_1) m_2(t,x) \rangle, \\
\       \mu(0,y) = 0.
 \end{cases}
 \end{equation}
Inequality \eqref{eq:energy2_init} immediately follows from Gronwall's inequality, keeping in mind that all norms are equivalent on the finite-dimensional vector space $\R^d(\mathcal{T} \times S)$.
\end{proof}

\begin{thm}[Existence]\label{thm1}
 Under Assumptions \ref{ass1} and \ref{ass2}, system \eqref{eq:dmfg} has a solution.
\end{thm}

\begin{proof}
We follow the proof of \cite[Thm.\@ 3.6]{BLP2022}.
We note first that by Assumption \ref{ass2}, the composed mapping $\fp \circ \v \circ \hjb$ is valued in $\mathcal{P}(\tilde{\mathcal{T}},S)$.
By Brouwer's fixed point theorem, it suffices to show that $\fp \circ \v \circ \hjb$ is continuous. The continuity of $\hjb$ and $\fp$ was established in Lemmas \ref{lm:hjb_stab}-\ref{lemma:cont_FP}. The continuity of $\v$ is deduced from the strong convexity of $\ell$, see step 2 of the proof of \cite[Thm.\@ 3.6]{BLP2022}.
\end{proof}

\subsection{Potential formulation}\label{sec:pdmfg}

Similarly to the case of continuous MFGs (see \cite{lasry2007mean,benamou2017} for example), the system \eqref{eq:dmfg} has a potential formulation. Consider the following optimal control problem:
\begin{equation} \label{pb:P}\tag{$P$}
\inf_{
\begin{subarray}{c}
m \in \mathcal{P}(\tilde{\mathcal{T}}, S) \\
v \in \mathbb{R}^d(\mathcal{T}\times S)
\end{subarray}
} J(m,v),
\quad \text{subject to: } (m,v) \in \mathcal{A},
\end{equation}
where the cost function $J$ and the set $\mathcal{A}$ are defined by
\begin{align*}
J(m,v)
= {} &
\Delta t \sum_{t \in \mathcal{T} } \sum_{x \in S} \ell[v](t,x)m(t,x) + \Delta t \sum_{t\in \mathcal{T}} F(t,m(t)) + \sum_{x\in S} g(x)m(T,x); \\
\mathcal{A} = {} & \left\{ (m,v) \in \mathcal{P}(\tilde{\mathcal{T}}, S)  \times \mathbb{R}^d(\mathcal{T}\times S) \,\big\vert \, m = \fp(v) ,\, \|v\|_{\infty,\infty}\leq D \right\}.
\end{align*}
Problem \eqref{pb:P} is a non-convex problem which can be made convex with the classical Benamou-Brenier transform (see \cite{benamou2017} for example) defined by
\begin{equation} \label{eq:BB}
    \chi\colon  (m,v) \in \mathcal{A} \mapsto
    (m,mv) \in \mathcal{P}(\tilde{\mathcal{T}}, S)  \times \mathbb{R}^d(\mathcal{T}\times S).
\end{equation}
Here $mv$ is the pointwise product of $m$ and $v$: $mv(t,x) \coloneqq m(t,x) v(t,x)$, for all $t \in \mathcal{T}$ and for all $x \in S$.
We set $\tilde{\mathcal{A}}= \chi(\mathcal{A})$ and consider the cost function $\tilde{J}$, defined by
\begin{equation*}
\tilde{J}(m,w)=\Delta t \sum_{t\in \mathcal{T}}\sum_{x\in S} \tilde{\ell}[m,w](t,x) + \Delta t \sum_{t\in \mathcal{T}} F(t,m(t)) + \sum_{x\in S} g(x)m(T,x),
\end{equation*}
where the function $\tilde{\ell}[m,w] \colon \mathcal{T}\times S \to \bar{\R}$ is defined by
\begin{equation*}
\tilde{\ell}[m,w](t,x)= \begin{cases}
\begin{array}{ll}
\ell\Big(t,x,\frac{w(t,x)}{m(t,x)} \Big)m(t,x), \quad & \text{if } m(t,x)\neq 0,\\
0, & \text{if } m(t,x)=w(t,x)= 0,\\
+ \infty, & \text{otherwise.}
\end{array}
\end{cases}
\end{equation*}
The new problem of interest is
\begin{equation} \label{pb:tP}\tag{$\tilde{P}$}
\inf_{
\begin{subarray}{c}
m \in \mathcal{P}(\tilde{\mathcal{T}}, S) \\
w \in \mathbb{R}^d(\mathcal{T}\times S)
\end{subarray}
} \tilde{J}(m,w),
\quad \text{subject to: } (m,w) \in \tilde{\mathcal{A}}.
\end{equation}
In the rest of the section, we investigate some properties of Problems \eqref{pb:P} and \eqref{pb:tP}, as well as their relationship with \eqref{eq:dmfg}

\begin{lem}\label{lm:equi}
It holds that $\textnormal{val}\eqref{pb:P} = \textnormal{val}\eqref{pb:tP}$.
\end{lem}

\begin{proof}
It follows from the definitions of $\mathcal{A}$, $J$, $\tilde{J}$, and $\chi$ that for any $(m,v) \in\mathcal{A}$, we have $\tilde{J}(\chi(m,v)) = J(m,v)$. The lemma follows immediately.
\end{proof}

We next discuss the convexity of Problem \eqref{pb:tP}.

\begin{lem}[Convexity] \label{lm:conv}
The set $\tilde{\mathcal{A}}$ is convex. The cost function $\tilde{J}(m,w)$ is convex.
\end{lem}

\begin{proof}
Take any $(m_1,w_1), (m_2,w_2)$ in $\tilde{\mathcal{A}}$ and any $\lambda\in (0,1)$. By the definition of $\tilde{\mathcal{A}}$, there exist $v_1,v_2 \in \mathbb{R}^d(\mathcal{T}\times S)$, such that $(m_i,v_i)\in\mathcal{A}$ and $w_i = m_iv_i$ for $i=1,2$. Let
\begin{align*}
 m & = \lambda m_1 + (1-\lambda)m_2, \\
 w & = \lambda w_1 +(1-\lambda)w_2,\\[0.2em]
 v(t,x) & =  \begin{cases} \
 0 , \qquad & \text{if }m_1(t,x)=m_2(t,x) =0, \\
\ \frac{\lambda m_1 v_1 + ( 1 - \lambda )m_2v_2}{\lambda m_1 + (1-\lambda)m_2}(t,x), \quad &\text{otherwise}.
 \end{cases}
\end{align*}
We can check that $(m,v)\in\mathcal{A}$ and that $(m,w) = \chi(m,v)$. The convexity of $\tilde{\mathcal{A}}$ follows.

The function $\tilde{J}$ is defined as the sum of three terms. The last one is linear, thus convex. The second one is also convex, by Lemma \ref{lm:F_bis}. Finally, for any $(t,x)$, the map $(m,w) \mapsto \tilde{\ell}[m,w](t,x)$ is convex (see \cite[Proposition 2.3]{combettes2018}). The convexity of $\tilde{J}$ follow.
\end{proof}

Given $m'\in \mathcal{P}(\tilde{\mathcal{T}},S)$, we consider the cost function $\tilde{J}_{m'}$, defined by
\begin{equation*}
\tilde{J}_{m'}(m,w) = \Delta t \sum_{t\in \mathcal{T}}\sum_{x\in S} \tilde{\ell}[m,w](t,x) + f(t,x,m'(t))m(t,x) + \sum_{x\in S} g(x)m(T,x),
\end{equation*}
for $(m,w) \in \mathcal{P}(\tilde{T},S) \times \R^d(\mathcal{T} \times S)$.
We will regard $\tilde{J}_{m'}$ as a partial linearization of $J$ around $m'$. We define the corresponding optimal control problem:
\begin{equation} \label{pb:ltP} \tag{$P_{m'}$}
\inf_{(m,w) \in \tilde{\mathcal{A}}} \ \tilde{J}_{m'}(m,w).
\end{equation}

\begin{lem} \label{lemma:convJ}
Let $(m_1,w_1)$ and $(m_2,w_2)$ be in $\tilde{\mathcal{A}}$. Then
\begin{equation*}
\tilde{J}(m_2,w_2)- \tilde{J}(m_1,w_1)
\geq
\tilde{J}_{m_1}(m_2,w_2)- \tilde{J}_{m_1}(m_1,w_1).
\end{equation*}
\end{lem}

\begin{proof}
This is an immediate consequence of the definitions of $\tilde{J}$, $\tilde{J}_{m'}$, and Lemma \ref{lm:F_bis}.
\end{proof}

The next lemma provides us with a solution to Problem \eqref{pb:ltP}.

\begin{lem}\label{lm:fund}
Let $m' \in \mathcal{P}(\tilde{\mathcal{T}},S)$.
Let us set
$\tilde{v}= \textnormal{\v} \circ \textnormal{\hjb}(m')$,
$\tilde{m}= \textnormal{\fp} \circ \textnormal{\v} \circ \textnormal{\hjb}(m')$,
and
$\tilde{w}= \tilde{m} \tilde{v}$.
Then $(\tilde{m},\tilde{w})$ is the unique solution to \eqref{pb:ltP}.
Moreover, for any $(m,w) \in \tilde{\mathcal{A}}$ and for any $v$ such that $w= mv$, it holds that
\begin{equation} \label{eq:fundamental_ineq}
\tilde{J}_{m'}(m,w) - \tilde{J}_{m'}(\tilde{m},\tilde{w}) \geq  \frac{\alpha}{2}\Delta t \sum_{t\in \mathcal{T}}\sum_{x\in S} \|(v-\tilde{v})(t,x)\|^2m(t,x).
\end{equation}
\end{lem}

\begin{proof}
The first inequality is proved in \cite[Sec.\@ 3.5, page 14]{BLP2022}. It is of similar nature to the fundamental equality of \cite{achdou2013mean}.
As a consequence of \eqref{eq:fundamental_ineq}, the pair $(\tilde{m},\tilde{w})$ is a solution to \eqref{pb:ltP}. It remains to prove uniqueness. Let $(\hat{m},\hat{w}) \in \tilde{\mathcal{A}}$ be a solution to \eqref{pb:ltP}. Let $\hat{v}$ be such that $\hat{w}= \hat{m} \hat{v}$. Then, by inequality \eqref{eq:fundamental_ineq}, we have
$\sum_{t \in \mathcal{T}} \sum_{x \in S} \| \hat{v}(t,x)- \tilde{v}(t,x) \|^2 \hat{m}(t,x)= 0$. It follows next that $(\hat{v}-\tilde{v})(t,x) \hat{m}(t,x)=0$ for all $(t,x)$. Applying Lemma \ref{lemma:cont_FP}, we immediately obtain that $\hat{m}= \tilde{m}$. Finally, we have $\hat{w}- \tilde{w}
= \hat{m} \hat{v}- \tilde{m}\tilde{v}
= \hat{m}(\hat{v}-\tilde{v})= 0$,
which concludes the proof of uniqueness.
\end{proof}

\begin{lem} \label{lem:uniqueness_all}
System \eqref{eq:dmfg} has a unique solution $(\bar{m},\bar{u},\bar{v})$. Moreover, $(\bar{m},\bar{w}) \coloneqq \chi(\bar{m},\bar{v})$ is the unique solution to \eqref{pb:tP}.
\end{lem}

\begin{proof}
Let $(m',u',v')$ be a solution to \eqref{eq:dmfg}. Let $w'= m'v'$. Combining Lemma \ref{lemma:convJ} and Lemma \ref{lm:fund}, we deduce that for any $(m,w) \in \tilde{\mathcal{A}}$,
\begin{equation*}
\tilde{J}(m,w)- \tilde{J}(m',w')
\geq
\tilde{J}_{m'}(m,w)- \tilde{J}_{m'}(m',w')
\geq 0.
\end{equation*}
Thus $(m',w')$ is a solution to \eqref{pb:tP}.

We next prove that $(m',w')$ is the unique solution to \eqref{pb:tP}. Let $(m,w)$ be a solution to \eqref{pb:tP}. The above inequality shows that $(m,w)$ is also a solution to \eqref{pb:ltP}. Thus by Lemma \ref{lm:fund}, $(m,w)=(m',w')$.

It remains to prove the uniqueness of the solution to \eqref{eq:dmfg}. Let $(m,u,v)$ be a solution to \eqref{eq:dmfg}. As was proved above, $(m,mv)$ is a solution to \eqref{pb:tP} and therefore $m= m'$. It follows that $u= \hjb(m)= \hjb(m')= u'$ and that $v= \v(u)= \v(u')= v'$, which concludes the proof.
\end{proof}

\section{Generalized Frank-Wolfe algorithm: the discrete case}
\label{sec:FW}

We investigate in this section the convergence of the GFW algorithm, applied to the (convex) potential problem \eqref{pb:tP}.
In this section, Assumptions \ref{ass1}-\ref{ass3} are supposed to be satisfied. We recall that \eqref{eq:dmfg} has a unique solution
$(\bar{u},\bar{v},\bar{m})$ and that by Lemma \ref{lem:uniqueness_all}, $(\bar{m},\bar{w}) = \chi(\bar{m},\bar{v})$ is the unique solution of problem \eqref{pb:tP}.

\subsection{Algorithm and convergence results}

We first define the mapping $\br \colon \mathcal{P}(\tilde{\mathcal{T}},S) \rightarrow \tilde{\mathcal{A}}$. Given $m' \in \mathcal{P}(\tilde{\mathcal{T}},S)$, we obtain $(\tilde{m},\tilde{w})= \br(m')$ by successively computing
\begin{equation*}
\tilde{v} = \v \circ \hjb (m'), \quad
\tilde{m} = \fp(\tilde{v}), \quad \text{and} \quad
\tilde{w}= \tilde{m}\tilde{v}.
\end{equation*}
We refer to $\br$ as the best-response mapping: given a prediction $m'$ of the equilibrium distribution of the agents, $\tilde{v}$ (as defined above) is the optimal feedback for the underlying optimal control problem and $\tilde{m}$ the resulting distribution.
As was demonstrated in Lemma \ref{lm:fund}, $\br(m')$ is also the unique solution to the linearized problem \eqref{pb:ltP}. This allows us to write the GFW algorithm for the resolution of \eqref{pb:tP} in the form of a best-response algorithm.

\medskip

\begin{algorithm}[H]
\SetAlgoLined
Initialization: $ (m^0, w^0) \in \tilde{\mathcal{A}} $\;
First iteration: $ (m^1,w^1)  = (\bar{m}^0,\bar{w}^0) = \br(m^0) $    \;
\For{$k= 1,2,\ldots$}{
\medskip
\textbf{Step 1: Resolution of the partial linearized problem.} \\
Set $(\bar{m}^k, \bar{w}^k) = \br(m^k) $\;
\medskip
\textbf{Step 2: Update.}\\
Choose $\lambda_k \in [0,1]$\;
Set $(m^{k+1},w^{k+1}) = (1-\lambda_k) (m^{k},w^{k}) + \lambda_k (\bar{m}^k, \bar{w}^k)$\;
\medskip
}
\caption{Generalized Frank-Wolfe Algorithm }\label{alg1}
\end{algorithm}

\medskip

Note that the choice of the stepsize $\lambda_k$ will be discussed in Proposition \ref{prop:convergence1}.
We introduce now three constants, $C_1$, $C_2$, and $C_3$, that will be used for the convergence analysis of Algorithm \ref{alg1}.
The constants $C_1$ and $C_2$ are defined by
\begin{equation} \label{eq:def_C_12}
C_1 = \sup_{\|v\|_{\infty,\infty}\leq D} \| \fp(v) \|_{\infty,2}^2
\quad \text{and}
\quad
C_2= \sup_{m\in \mathcal{P}(\tilde{\mathcal{T}},S)} \|\fp \circ \v \circ \hjb (m) \|_{\infty,\infty}.
\end{equation}
The finiteness of $C_1$ and $C_2$ follows from the compactness of $\mathcal{P}(\tilde{\mathcal{T}},S)$ and from the continuity of the three mappings $\hjb$, $\v$, and $\fp$. Note that
\begin{equation} \label{eq:def_C_12bis}
C_1 =  \sup_{(m,v)\in\mathcal{A}}\| m \|_{\infty,2}^2 = \sup_{(m,w)\in \tilde{\mathcal{A}}}\| m \|_{\infty,2}^2.
\end{equation}
By Lemma \ref{lemma:cont_FP}, there exists a constant $C_3>0$ such that for any $(m_1,v_1)$ and $ (m_2,v_2)$ in $\mathcal{A}$,
\begin{equation}\label{eq:energy2}
     \| m_1 - m_2\|^2_{\infty,2} \leq C_3\Delta t \sum_{t\in \mathcal{T}}\sum_{x\in S} \|(v_1 - v_2)m_1(t,x)\|^2. 
\end{equation}
We next introduce three other constants, $D_1$, $D_2$, and $c$, defined by:
\begin{equation}\label{eq:C}
\begin{split}
      D_1 = C_1 L_f |S|^{1/2}, \quad  D_2 = (L_f + |S|^{1/2}) \sqrt{\frac{2C_2 C_3}{\alpha}},\quad   c = \max\left\{1 - \frac{\alpha}{4C_2C_3L_f|S|^{1/2}} \, ,\, \frac{1}{2}\right\}.
\end{split}
\end{equation}

\begin{prop} \label{prop:convergence1}
We consider the sequence $(m_k,w_k)_{k \geq 1}$ generated by Algorithm \eqref{alg1}. 
\begin{enumerate}
\item \emph{Sublinear rate.} Assume that $\lambda_k = 2/(k+2)$, for all $k \geq 1$. Then,
\begin{equation} \label{eq:conv1_2}
\tilde{J}(m^k,w^k) - 
\tilde{J}(\bar{m},\bar{w})
\leq \frac{ 8D_1}{k}, \qquad \forall k\geq 1.
\end{equation}
\item \emph{Linear rate.} Assume that
\begin{equation}\label{eq:lambda}
\lambda_k =  \min\left\{\frac{\tilde{J}_{m^k}(m^k,w^k) - \tilde{J}_{m^k}(\bar{m}^k,\bar{w}^k)}{L_f |S|^{1/2}\|m^k-\bar{m}^k\|_{\infty,2}^2} \, , \, 1\right\},
\end{equation}
for all $k \geq 1$. Then,
\begin{equation}\label{eq:conv2_2}
\tilde{J}(m^k,w^k) - 
\tilde{J}(\bar{m},\bar{w}) \leq 4 D_1 c^{k}, \qquad \forall k \geq 1.
\end{equation}
\end{enumerate}
\end{prop}

The following subsection is dedicated to the proof of Proposition \ref{prop:convergence1}. The used stepsize rule \eqref{eq:lambda} is motivated in the proof of convergence (see \eqref{eq:justification}).

\subsection{Convergence analysis} \label{sec:convergence}

\begin{lem}[A priori bounds]\label{lm:m^k}
For any $k\geq 1$, we have $(m^k,w^k) \in \tilde{\mathcal{A}}$. As a consequence,
\begin{equation} \label{eq:upper_bound}
\|m^k\|_{\infty,2}^2 \leq C_1 \quad
\text{and} \quad
\|m^k\|_{\infty,\infty}\leq C_2.
\end{equation}
\end{lem}

\begin{proof}
Using the definitions of $C_1$ and $C_2$ (given in \eqref{eq:def_C_12}) and using \eqref{eq:def_C_12bis}, we have that
\begin{equation} \label{eq:new_inequality}
\| \bar{m}^k \|_{\infty,2}^2 \leq C_1,
\| \bar{m}^k \|_{\infty,\infty} \leq C_2,
\quad \text{and} \quad
(\bar{m}^k,\bar{w}^k) \in \tilde{\mathcal{A}},
\end{equation}
for any $k \geq 1$. To conclude the proof of the lemma, it suffices to observe that for any $k \geq 1$, $({m}^k, {w}^k)$ is a convex combination of $(\bar{m}^{\kappa}, \bar{w}^{\kappa}) \in \tilde{\mathcal{A}}$ for $0\leq \kappa \leq k-1$. Then we have $(m^k,w^k) \in \tilde{\mathcal{A}}$, since $\tilde{\mathcal{A}}$ is convex, by Lemma \ref{lm:conv}. Inequality \eqref{eq:upper_bound} follows from \eqref{eq:new_inequality} and from the triangle inequality.
\end{proof}

\begin{lem} \label{lm:F}
For any $t\in \mathcal{T}$ and for any $m_1$ and $m_2$ in $\mathcal{P}(S)$, we have
\begin{equation}
F(t,m_2) \leq F(t,m_1) + \sum_{x\in S} f(t,x, m_1 ) (m_2(x) - m_1(x)) + \frac{L_f|S|^{1/2}}{2} \|m_2-m_1\|_{2}^2. \label{eq:quadratic}
\end{equation}
\end{lem}

\begin{proof}
Combining Assumption \ref{ass3} and the Lipschitz-continuity of $f$ (Assumption \ref{ass1}), we deduce that
\begin{equation*}
F(t,m_2) \leq F(t,m_1) + \sum_{x\in S} f(t,x, m_1 ) (m_2(x) - m_1(x)) + \frac{L_f}{2} \|m_2-m_1\|_{2}\|m_2 - m_1\|_{1}.
\end{equation*}
By H\"older's inequality, we have that $\|m_2-m_1\|_1\leq |S|^{1/2} \|m_2-m_1\|_2$. Inequality \eqref{eq:quadratic} follows.
\end{proof}

Algorithm \ref{alg1} generates two sequences $(m^k,w^k)_{k \geq 0}$ and $(\bar{m}^k,\bar{w}^k)_{k \geq 0}$ in $\tilde{\mathcal{A}}$. For the analysis, we need to fix two sequences $(v^k)_{k \geq 0}$ and $(\bar{v}^k)_{k \geq 0}$ such that $w^k=m^k v^k$ and $\bar{w}^k=\bar{m}^k \bar{v}^k$. We introduce the following six sequences of positive numbers:
\begin{align*}
& \delta_k = \|m^k-\bar{m}\|_{\infty,2}^2,
& \bdel_k & = \|m^k-\bar{m}^k\|_{\infty,2}^2;\\[0.8em]
& \gamma_k = \tilde{J}(m^k,w^k) - \tilde{J}(\bar{m},\bar{w}), \qquad
& \bga_k & = \tilde{J}_{m^k}(m^k,w^k) - \tilde{J}_{m^k}(\bar{m}^k,\bar{w}^k);\\[0.8em]
& \epsilon_k = \Delta t \sum_{t\in \mathcal{T}}\sum_{x\in S} \|v^k - \bar{v}\|^2m^{k}(t,x), &
\beps_k & = \Delta t \sum_{t\in \mathcal{T}}\sum_{x\in S} \|v^k - \bar{v}^k\|^2m^{k}(t,x).
\end{align*}
The following lemma establishes various relationships between these sequences, independent of the choice of stepsize $(\lambda_k)_{k \geq 1}$. For convenience, we fix $\lambda_0= 1$ and make use of the fact that $(m^1,w^1)= (1-\lambda_0)(m^0,w^0) + \lambda_0 (\bar{m}^0,\bar{w}^0)$.

\begin{lem} \label{lm:FW}
For any choice of stepsizes $(\lambda_k)_{k \geq 1}$, we have $\bdel_k \leq 4C_1$, for any $k \geq 1$. Moreover we have that
\begin{equation}  \label{eq:classical_bound}
\gamma_k \leq \bga_k \qquad \text{and} \qquad
 \gamma_{k+1} \leq \gamma_k -\lambda_k \bga_k + \lambda_k^2 \frac{L_f|S|^{1/2}}{2} \bdel_k.
\end{equation}
We also have the following estimates:
\begin{equation} \label{eq:stab_bound}
\frac{\alpha}{2C_2C_3} \delta_k \leq \frac{\alpha}{2} \epsilon_k \leq \gamma_k \qquad \text{and} \qquad
\frac{\alpha}{2C_2C_3} \bdel_k \leq \frac{\alpha}{2} \beps_k \leq \bga_k.
\end{equation}
\end{lem}

\begin{proof}
\textbf{Step 1.}
The inequality $\bar{\delta}_k = \|m^k-\bar{m}^k\|_{\infty,2}^2 \leq 4 C_1$ follows from the bounds $\|m^k \|_{\infty,2}^2 \leq C_1$ and $\|\bar{m}^k\|_{\infty,2}^2 \leq C_1$ obtained in Lemma \ref{lm:m^k}.

\textbf{Step 2.}
We next prove that $\gamma_k \leq \bar{\gamma}_k$. Recalling that $(\bar{m}^k,\bar{w}^k)$ minimizes $\tilde{J}_{m^k}(\cdot)$ over $\tilde{\mathcal{A}}$, we obtain that
\begin{equation*}
    \bga_k  =   \tilde{J}_{m^k}(m^k,w^k) - \tilde{J}_{m^k}(\bar{m}^k,\bar{w}^k) \geq \tilde{J}_{m^k}(m^k,w^k) - \tilde{J}_{m^k}(\bar{m},\bar{w}).
\end{equation*}
Using next Lemma \ref{lemma:convJ}, we deduce that
$\bga_k \geq J(m^k,w^k) - J(\bar{m},\bar{w})
 = \gamma_k$.
Let us prove the upper bound of $\gamma_{k+1}$. Since $\tilde{\ell}$ is convex (see the proof of Lemma \ref{lm:conv}), we have
\begin{equation*}
    \tilde{\ell}[m^{k+1},w^{k+1}](t,x) \leq (1-\lambda_k)\tilde{\ell}[m^k,w^k](t,x)  + \lambda_k \tilde{\ell}[\bar{m}^k,\bar{w}^k](t,x) .
\end{equation*}
Moreover, by Lemma \ref{lm:F}, we have
\begin{equation*}
      F(t,m^{k+1}(t)) \leq  F(t,m^k(t)) + \lambda_k\sum_{x\in S} f(t,x, m^k(t) ) (\bar{m}^k(t,x) - m^k(t,x)) + \frac{\lambda_k^2L_f|S|^{1/2}}{2} \bdel_k.
\end{equation*}
Then \eqref{eq:classical_bound} follows from the above two inequalities and the definitions of $\tilde{J}$, $\tilde{J}_{m^k}$, $\gamma_k$, and $\bar{\gamma}_k$.

\textbf{Step 3.} Let us prove \eqref{eq:stab_bound}. From the definition of $C_3$ and Lemma \ref{lm:m^k} we have that
\begin{align*}
\delta_k \leq {} & C_3 \delta t \sum_{t \in \mathcal{T}}
\sum_{x \in S} \| (v_k - \bar{v})(t,x) m_k(t,x) \|^2 \\
\leq {} & C_2 C_3 \delta t \sum_{t \in \mathcal{T}}
\sum_{x \in S} \| (v_k - \bar{v})(t,x) \|^2 m_k(t,x) = C_2 C_3 \epsilon_k.
\end{align*}
Moreover, using Lemmas \ref{lemma:convJ} and \ref{lm:fund}, we obtain that
\begin{equation*}
\gamma_k \geq \tilde{J}_{\bar{m}}(m^k) - \tilde{J}_{\bar{m}}(\bar{m})
\geq \frac{\alpha}{2} \epsilon_k.
\end{equation*}
We prove in a similar fashion that $\bar{\delta}_k \leq C_2 C_3 \bar{\epsilon}_k$ and that $\bar{\gamma}_k \geq \frac{\alpha}{2} \bar{\epsilon}_k$.
\end{proof}

\begin{lem}\label{lm:u_m}
In Algorithm \ref{alg1}, let $u^k =  \textnormal{\hjb}(m^k)$. Then, whatever the choice of stepsizes $(\lambda_k)_{k \geq 1}$, we have
\begin{equation*}
\|u^k - \bar{u}\|_{\infty,\infty} + \|m^k - \bar{m}\|_{\infty,1}  \leq D_2\sqrt{\gamma_k}.
\end{equation*}
\end{lem}

\begin{proof}
By Cauchy-Schwarz inequality and Lemma \ref{lm:FW}, we have
\begin{equation*}
\| m^k - \bar{m} \|_{\infty,1}
\leq \sqrt{|S|} \sqrt{\delta_k} \leq \sqrt{|S|} \sqrt{\frac{2C_2C_3}{\alpha}} \sqrt{\gamma_k}.
\end{equation*}
By Lemma \ref{lm:hjb_stab}, we have
$\| u^k - \bar{u} \|_{\infty,\infty}
\leq L_f \sqrt{\delta_k}$. Recalling the definition of $D_2$, we obtain the announced result.
\end{proof}

\begin{proof}[Proof of Proposition \ref{prop:convergence1}]
\textbf{Sublinear case}.
Combining the two inequalities in \eqref{eq:classical_bound} and using the bound $\bar{\delta}_k \leq 4C_1$, we obtain that
\begin{equation*}
\gamma_{k+1} 
\leq (1-\lambda_k)\gamma_k + 2\lambda_k^2 C_1L_f|S|^{1/2}
=  (1-\lambda_k)\gamma_k + 2\lambda_k^2 D_1.
\end{equation*}
In particular, for $k=0$, we have $\lambda_0=1$, therefore $\gamma_1 \leq 2D_1$. It is then easy to prove by induction the inequality $\gamma_k \leq 8D_1/k$ (see \cite[Thm.\@ 1]{Jaggi2013} for example).

\textbf{Linear case}. Note first that the stepsize rule \eqref{eq:lambda} writes
\begin{equation*}
\lambda_k =  \min\left\{ \frac{\bar{\gamma}_k}{L_f |S|^{1/2} \bar{\delta}_k } \, , \, 1\right\}.
\end{equation*}
It is easy to verify that $\lambda_k$ minimizes the upper-bound \eqref{eq:classical_bound}, that is to say:
\begin{equation} \label{eq:justification}
\lambda_k = \underset{\lambda \in [0,1]}{\text{argmin}}
-\lambda \bar{\gamma}_k + \lambda^2 \frac{L_f |S|^{1/2}}{2\bar{\delta}_k}.
\end{equation}
Let $k\geq 1$. We consider the following two cases:
\begin{enumerate}
\item If $\bar{\gamma}_k \geq L_f |S|^{1/2}\bar{\delta}_k$, then $\lambda_k = 1$. We deduce from \eqref{eq:classical_bound} that
\begin{equation*}
\gamma_{k+1} \leq \gamma_k - \bga_k + \frac{L_f|S|^{1/2}}{2} \bdel_k \leq \gamma_k - \frac{\bga_k}{2} \leq \frac{\gamma_k}{2} \leq c \gamma_k,
\end{equation*}
where the second inequality follows from the assumption $\bar{\gamma}_k \geq L_f |S|^{1/2}\bar{\delta}_k$, the third one from the inequality $\gamma_k \leq \bar{\gamma}_k$ and the last one from the fact that $c \geq 1/2$.
\item If $\bar{\gamma}_k \leq L_f |S|^{1/2}\bar{\delta}_k $, then
$\lambda_k= \frac{\bar{\gamma}_k}{L_f |S|^{1/2} \bar{\delta}_k }$.
We deduce from \eqref{eq:classical_bound} and from the inequality $\gamma_k \leq \bar{\gamma}_k$ that
\begin{equation*}
\gamma_{k+1}
\leq \gamma_k - \frac{\bar{\gamma}_k}{2 L_f |S|^{1/2}\bar{\delta}_k}\bga_k
\leq \left( 1 - \frac{\bar{\gamma}_k}{2L_f |S|^{1/2}\bar{\delta}_k} \right) \gamma_k.
\end{equation*}
By \eqref{eq:stab_bound}, we know that $\bar{\gamma}_k /\bar{\delta}_k \geq \frac{\alpha}{2C_2C_3}$. It follows that
\begin{equation*}
\gamma_{k+1}
\leq \left(1 - \frac{\alpha}{4C_2C_3L_f|S|^{1/2}}\right)\gamma_k \leq c\gamma_k.
\end{equation*}
\end{enumerate}
It follows that $\gamma_k \leq \gamma_1 c^{k-1}$ for all $k\geq 1$. Since $c^{-1} \leq 2$ and $\gamma_1 \leq 2D_1$, estimate \eqref{eq:conv2_2} holds true.
\end{proof}

\section{Mesh-independent convergence of the GFW algorithm for second-order MFGs}
\label{sec:mesh-indep}

We establish in this section two mesh-independence principles for the GFW algorithm, applied to a discretization of \eqref{eq:mfg} with the $\theta$-scheme proposed in \cite{BLP2022}. First, we recall the theta-scheme and the main convergence result of \cite{BLP2022}. Then, we show that the potential structure of the continuous \eqref{eq:mfg} is preserved by the theta-scheme at the discrete level. Therefore, Algorithm \ref{alg1} can be applied to the theta-scheme. The mesh-independence principles are stated in Theorems \ref{thm:main2} and \ref{thm:main3} and demonstrated in Subsections \ref{subsec:proof_sublin} and \ref{subsec:proof_lin}.

\subsection{The theta-scheme and error estimates} \label{section:theta}
 
In this subsection, we recall the theta-scheme of the continuous system \eqref{eq:mfg} investigated in \cite{BLP2022} and its main result.  Let us define
\begin{equation} \label{eq:def_D}
\mathcal{D} =  \Big\{ \mu \in \mathbb{L}^2(\mathbb{T}^d) \, \vert \, \mu\geq 0, \int_{\mathbb{T}^d} \mu(x) dx =1  \Big\}.
\end{equation}
We make the following assumptions on the data function $\ell^c \colon Q\times \mathbb{R}^d \to \R$, $g^c \colon \mathbb{T}^d \to \R$, $f^c \colon Q \times \mathcal{D} \to \R$, and $m^c \colon \mathbb{T}^d \to \R$ of the continuous model \eqref{eq:mfg}. 

\begin{ass}\label{ass:continuous}
\begin{enumerate}
\item \emph{Regularity}. The function $\ell^c$ is continuously differentiable with respect to $v$. There exist positive constants $L_{\ell}^c$, $L_g^c$, and $L_f^c$ such that for any $(t,x) \in Q$, for any $v \in \R^d$, and for any $m \in \mathcal{D}$,
\begin{itemize}
\item $\ell^c(\cdot,x,v)$, $\ell^c(t,\cdot,v)$, and $\ell_v^c(\cdot,x,v)$ are $L_{\ell}^c$-Lipschitz continuous
\item $g^c$ is $L_g^c$-Lipschitz continuous
\item $f^c(\cdot,x,m)$, $f^c(t,\cdot,m)$, and $f^c(t,x,\cdot)$ are $L_f^c$-Lipschitz continuous (with respect to the $\| \cdot \|_{\mathbb{L}^2}$-norm for the third variable).
\end{itemize}
\item \emph{Strong convexity}. There exists $\alpha^c > 0$ such that for any $(t,x) \in Q$, $\ell^c(t,x,\cdot)$ is strongly convex with modulus $\alpha^c$, i.e.
\begin{equation*}
 \ell^c(t,x,v_2) \geq \ell^c(t,x,v_1) + \langle \ell_v^c(t,x,v_1) , v_2-v_1\rangle + \frac{\alpha^c}{2}\|v_2-v_1\|^2,
 \quad
 \forall v_1, v_2 \in \R^d.
\end{equation*}
\item \emph{Monotonicity.} The function $f^c$ is monotone, i.e., for any $t \in [0,T]$, for any $m_1$ and $m_2 \in \mathcal{D}$,
\begin{equation*}
        \int_{ \mathbb{T}^d} \Big( f^c(t,x',m_1) - f^c(t,x',m_2)\Big) \big(m_1(x') - m_2(x')\big) dx' \geq 0.
\end{equation*}
\end{enumerate}
\end{ass}

\begin{ass}\label{ass:sol+}
The continuous mean field game \eqref{eq:mfg} admits a unique solution $(u^{*}, v^{*}, m^{*})$, with $u^{*}, m^{*} \in \mathcal{C}^{1+r/2,2+r}(Q) $ and $v^{*}\in \mathcal{C}^r(Q) \cap \mathbb{L}^{\infty}([0,1]; \mathcal{C}^{1+r}(\mathbb{T}^d))$, where $r\in (0,1)$.
\end{ass}

Note that more explicit assumptions on the problem data allows to verify Assumption \ref{ass:sol+}, see \cite[Appendix B]{BLP2022}.

The time and space discretization parameters are denoted $\Delta t>0$ and $h>0$. As in the discrete model, we suppose that $\Delta t = 1/T$ and that $h= 1/N$, for two integers $T$ and $N$.
The discrete state space is defined as
\begin{equation*}
S = \mathbb{T}^d/h\mathbb{Z}^d = \big\{ (i_1,i_2,\ldots,i_d) h \; \mid \;  i_1,\ldots,i_d \in \mathbb{Z}/N\mathbb{Z} \big\}.
\end{equation*}
Given $x\in \mathbb{T}^d$, let us set $B_h(x) = \prod_{i=1}^n [x-he_i/2,\, x +he_i/2)$. Given $y \in \mathbb{T}^d$, we denote by $x_h[y]$ the unique point $x$ in $S$ such that $y \in B_h(x)$.
We consider the mappings $\mathcal{I}_h \colon \mathbb{R}(\mathbb{T}^d) \to \mathbb{R}(S)$ and  $\mathcal{R}_h \colon \mathbb{R}(S) \rightarrow \mathbb{R}(\mathbb{T}^d)$, defined as follows: For any $m^c\in \mathbb{R}(\mathbb{T}^d)$ and for any $m \in \mathbb{R}(S)$, 
\begin{align*}
\mathcal{I}_h (m^c) (x) & = \int_{B_h(x)} m^c(y)dy,  \quad &&\forall x\in S;\\
\mathcal{R}_h (m) (y) & = \frac{m(x_h[y])}{h^d}, && \forall y \in \mathbb{T}^d.
\end{align*}
We consider the constant $M>0$, defined by
\begin{equation} \label{eq:cons_M}
M = \frac{1}{\alpha^c}
\Big(
2 \max_{(t,x) \in Q}
\| \ell^c_v(t,x,0) \|
+ \sqrt{d} (L_{\ell}^c + L_f^c + L_g^c)
\Big).
\end{equation}
Note that $M$ is independent of $\Delta t$ and $h$.
We define the truncated running cost $\hat{\ell}^c$ as follows:
\begin{equation*}
\hat{\ell}^c(t,x,v) = \begin{cases} \
\ell^c(t,x,v), \quad & \text{if }\|v\|\leq M, \\
\ +\infty, & \text{otherwise}.
\end{cases}
\end{equation*}
The discrete counterparts of $\ell^c$, $g^c$, $f^c$, and $m_0^c$ are defined as follows: For any $t \in \mathcal{T}$, for any $x \in S$, for any $v \in \R^d$, for any $m \in \mathcal{P}(S)$,
\begin{equation}\label{eq:grid}
\begin{array}{rl}
    \ell(t,x,p) = &\! \! \! \hat{\ell}^c(t\Delta t, x, v) \\[0.5em]
     g(x) = & \! \! \! g^c(x), \\[0.5em]
   f(t,x,m) =  & \! \! \! {\displaystyle \frac{1}{h^d}\int_{B_h(x)} } f^c\big(t\Delta t,y,\mathcal{R}_h(m) \big) dy, \\[1.2em]
   m_0(x) = &\! \! \! \mathcal{I}_h (m_0^c)(x)   .
\end{array}
\end{equation}
We denote by $H\colon \mathcal{T} \times S \times \R^d$ the associated Hamiltonian, defined by
\begin{equation} \label{eq:discrete_hamiltonian}
H (t,x, p) =\sup_{v \in \R^d} -\langle p ,v \rangle -\ell(t,x,v).
\end{equation}
 
Let us now discetize the differential operators appearing in the continuous MFG system. Let $(e_i)_{i=1,\ldots,d}$ denote the canonical basis of $\mathbb{R}^d$. The discrete Laplace, gradient and divergence operators for the centered finite-difference scheme are defined as follows:
\begin{align*}
    & \Delta_h \mu(x) = \sum_{i=1}^d \frac{  \mu(x+he_i) + \mu(x-he_i)  -2 \mu(x)}{h^2}, & \forall\;\mu\in \mathbb{R}(S), \; \forall \; x\in S,\\
    & \nabla_h \mu(x) =  \Big( \frac{\mu(x+he_i)- \mu(x-he_i)}{2h} \Big)_{i=1}^d, & \forall\;\mu\in \mathbb{R}(S), \; \forall \; x\in S,\\
    & \text{div}_h \mu (x) = \sum_{i=1}^d \frac{\mu_i(x+he_i) - \mu_i(x-he_i)}{2h}, & \forall \; \mu\in \mathbb{R}^d(S), \; \forall \; x \in S, 
\end{align*}
where $\mu_i$ is the $i$-th coordinate of $\mu$.

Finally, we fix $\theta \in (1/2,1)$. The theta-scheme for \eqref{eq:mfg}, introduced in \cite[Sec.\@ 2.3]{BLP2022}, writes:
\begin{equation}\label{eq:theta_mfg}\tag{Theta-mfg}
\begin{cases}
\begin{array}{cll}
\mathrm{(i)} \ & u =  \hjbt (m) ,\\[0.3em]
\mathrm{(ii)} \  & v = \vt (u),\\[0.3em]
\mathrm{(iii)} \ & m = \fpt (v),
\end{array}
\end{cases}
\end{equation}
where the Hamilton-Jacobi-Bellman mapping $\hjbt\colon \mathcal{P}(\tilde{\mathcal{T}},S) \to  \mathbb{R}(\tilde{\mathcal{T}}\times S),\, m\mapsto u $, is defined by
\begin{equation*}
        \begin{cases}
        \ -\frac{u(t+1,x) - u(t+1/2,x)}{\Delta t} - \theta \sigma \Delta_h u(t+1/2,x) = 0 , \\[0.6em]
         \ -\frac{u(t+1/2,x) - u(t,x)}{\Delta t} - (1-\theta)\sigma \Delta_h u(t+1/2, x) +  H [\nabla_h u ( \cdot +1/2 , \cdot ) ](t,x) = f(t,x,m(t)), \\[0.6em]
        \  u(T,x) = g(x),
\end{cases}
\end{equation*}
the optimal control mapping $\vt \colon \mathbb{R}(\tilde{\mathcal{T}}\times S)\to \mathbb{R}^d(\mathcal{T}\times S), \, u\mapsto v,$ is defined by
\begin{equation*}
        \begin{cases}
        \ -\frac{u(t+1,x) - u(t+1/2,x)}{\Delta t} - \theta \sigma \Delta_h u(t+1/2,x) = 0, \\[0.6em]
        \ v(t,x) = -H_p(t,x,\nabla_h u(t+1/2,x)),
\end{cases}
\end{equation*}
and the Fokker-Planck mapping $\fpt \colon \mathbb{R}^d(\mathcal{T}\times S) \to \mathbb{R}(\tilde{\mathcal{T}}\times S) ,\, v\mapsto  m$, is defined by
\begin{equation*}
\begin{cases}
        \ \frac{m(t+1,x) - m(t,x)}{\Delta t} - \theta\sigma \Delta_h m(t+1,x) - (1-\theta)\sigma \Delta_h m(t,x) +\text{div}_h \big(vm(t,x) \big) = 0 ; \\[0.6em]
         \ m(0,x) = m_0(x).
\end{cases}
\end{equation*}
Let us briefly motivate the theta-scheme. For the Fokker-Planck equation, a Crank-Nicolson discretization of the Laplace operator and an explicit discretization of the first-order term are utilized. An adjoint scheme is used for the HJB equation.

Given $u(t+1,\cdot)$, one first needs to solve an implicit scheme for the heat equation, corresponding to a diffusion equal to $\theta \sigma$. This yields the intermediate function $u(t+1/2,\cdot)$, which is not ``saved" in the output $\hjbt$. Then $u(t,\cdot)$ is obtained by solving an explicit scheme, containing the first order term and a diffusion equal to $(1-\theta) \sigma$.

\begin{rem}
The evaluation of the mapping $\hjbt$, which is a crucial step in the GFW algorithm, requires to solve successively linear equations (for the implicit part) and explicit equations. It is therefore easier to implement than fully implicit schemes, which would require to solve general non-linear equations (with a policy iteration algorithm, for example).
\end{rem}

Let us consider the following CFL condition:
\begin{equation}\label{cond:CFL}\tag{CFL}
\Delta t \leq \frac{h^2}{2d (1-\theta)\sigma} ,\qquad h \leq \frac{2(1-\theta)\sigma}{M}.
\end{equation}

\begin{lem}
Let Assumption \ref{ass:continuous} and condition \eqref{cond:CFL} hold true. Then the system \eqref{eq:theta_mfg} is a particular case of \eqref{eq:dmfg}, with $\ell$, $f$, $m_0$ and $g$ defined by \eqref{eq:grid}. Furthermore, Assumptions \ref{ass1}-\ref{ass2} hold with the constants
\begin{equation*}
D = M, \quad
L_{f} = L_f^c h^{-d/2},
\quad \text{and}
\quad \alpha = \alpha^c.
\end{equation*}
\end{lem}

\begin{proof}
See \cite[Lem.\@ 4.1, Lem.\@ 4.2, Thm.\@ 4.4]{BLP2022}.
\end{proof}

Let us emphasize that the well-posedness of the mappings $\hjbt$, $\vt$, and $\fpt$ is a corollary of the fact that \eqref{eq:theta_mfg} is a discrete MFG. The explicit formulas for $\pi_0$ and $\pi_1$ are provided in \cite[Lemma 4.1]{BLP2022}.
We define the traces  $u_h^* \in \R(\tilde{\mathcal{T}} \times S)$ and $m_h^* \in \mathcal{P}(\tilde{\mathcal{T}},S)$ of the continuous solution $(u^*,m^*)$ as follows:
\begin{equation} \label{eq:traces}
u_h^*(t,x)= u^*(t\Delta,x) \quad \text{and} \quad m_h^*(t,\cdot)= \mathcal{I}_h(m^*(t\Delta t, \cdot)),
\quad
\forall t \in \tilde{\mathcal{T}}, \forall x \in S.
\end{equation}
The main result of \cite{BLP2022} is the following error estimate.

\begin{thm}\label{thm:main1}
 Let $1/2 < \theta <1$. Let Assumptions \ref{ass:continuous}-\ref{ass:sol+} and condition \eqref{cond:CFL} hold true.
Then \eqref{eq:theta_mfg} admits a unique solution $(u_h, v_h, m_h)$. There exists $C^{*}$ independent of $\Delta t$ and $h$ such that
\begin{equation*}
\|u_h - u_h^{*} \|_{\infty,\infty}  + \|m_h - m_h^* \|_{\infty,1} \leq C^{*} h^r.
\end{equation*}
\end{thm}

\begin{proof}
See \cite[Thm.\@ 2.10]{BLP2022}.
\end{proof}

\subsection{GFW algorithm for the theta-scheme and main results}

Let us first give the assumption on the potential structure of $f$.

\begin{ass}\label{ass:potential}
There exists a function $F^c \colon [0,1]\times \mathcal{D} \to \mathbb{R}^d$ such that \eqref{eq:potential} holds.
\end{ass}

The following lemma shows that the discretization of $f^c$ proposed in \eqref{eq:grid} preserves the potential structure of the MFG system at the discrete level.

\begin{lem}\label{lm:ass}
Let Assumption \ref{ass:potential} hold true. Then, $f$ defined by \eqref{eq:grid} satisfies Assumption \ref{ass3}, with the primitive function $F$ defined by
\begin{equation*}
F (t,m) = F^c(t\Delta t,\mathcal{R}_h(m)).
\end{equation*}
\end{lem}

\begin{proof}
Taking any $(t,m_1,m_2) \in \mathcal{T}\times \mathcal{P}(S)^2$, we have that
\begin{equation*}
\begin{split}
F(t,m_1) - F(t,m_2) & = F^c(t\Delta t,\mathcal{R}_h(m_1)) - F^c(t \Delta t,\mathcal{R}_h(m_2)) \\
& =\int_{0}^1 \int_{x\in \mathbb{T}^d} f^c(t\Delta t,x, \mathcal{R}_h(m_1 + s(m_2-m_1))) (\mathcal{R}_h(m_2)(x) - \mathcal{R}_h(m_1)(x))dx ds\\
& = \int_{0}^1 \sum_{x\in S} \frac{1}{h^d}\int_{y\in B_h(x)} f^c(t\Delta t,y, \mathcal{R}_h(m_1 + s(m_2-m_1))) dy (m_2(x)-m_1(x)) ds \\
& = \int_{0}^1 \sum_{x\in S} f(t,x,m_1+s(m_2-m_1)) (m_2(x)-m_1(x)) ds,
\end{split}
\end{equation*}
where the second equality follows from Assumption \ref{ass:potential} and the linearity of operator $\mathcal{R}_h$, the third equality comes from the fact that $\mathcal{R}_h(m)$ is constant in $B_h(x)$, and the last equality derives from the definition of $f$.
\end{proof}

Following the definition of $\br$ in Sec.\@ \ref{sec:FW}, we define the best response mapping for \eqref{eq:theta_mfg}: for any $m'\in \mathcal{P}(\tilde{\mathcal{T}},S)$,
\begin{equation*}
\brt (m') = \chi (\tilde{m},\tilde{v}), \qquad \text{where }  \tilde{v} = \vt \circ \hjbt (m') \quad \text{and} \quad \tilde{m} = \fpt(\tilde{v}).
\end{equation*}
The GFW algorithm writes as follows.

\medskip

\begin{algorithm}[H]
\SetAlgoLined
Initialization: $ m_h^0 \in \mathcal{P}(\tilde{\mathcal{T}},S)$\;
First iteration: $ (m_h^1,w_h^1) = \brt(m_h^0) $    \;
\For{$k= 1,2,\ldots$}{
\medskip
\textbf{Step 1: Resolution of the partial linearized problem.} \\
Set $(\bar{m}_h^k, \bar{w}_h^k) = \brt(m_h^k)$\;
\medskip
\textbf{Step 2: Update.}\\
Choose $\lambda_k \in [0,1]$\;
Set $(m_h^{k+1},w_h^{k+1}) = (1-\lambda_k) (m_h^{k},w_h^{k}) + \lambda_k (\bar{m}_h^k, \bar{w}_h^k)$\;
\medskip
}
\caption{Generalized Frank-Wolfe Algorithm for \eqref{eq:theta_mfg}}\label{alg2}
\end{algorithm}

\medskip

From now on, all notations introduced in Sections \ref{sec:dmfg} and \ref{sec:FW} for general discrete MFGs will be restricted to \eqref{eq:theta_mfg}, without the adjunction of the subscript $\theta$. For example, we will denote by $\gamma_k$ the optimality gap in Algorithm \ref{alg2}. The following result is our first mesh-independence principle. Recall that $(u_h,m_h)$ and $(u_h^*,m_h^*)$ have been introduced in Theorem \ref{thm:main1}.

\begin{thm} [Sublinear rate]\label{thm:main2}
Let Assumptions \ref{ass:continuous}-\ref{ass:potential} and condition \eqref{cond:CFL} hold true. In Algorithm \ref{alg2}, take $\lambda_k = 2/(k+2)$, for any $k \geq 1$. Then there exist two constants $C_{\theta}$ and $\bar{C}_{\theta}$, both independent of $\Delta t$ and $h$, such that
\begin{equation} \label{eq:sublin_rate}
\gamma_{k} \leq \frac{C_{\theta}}{k}, \quad \forall k \geq 1,
\end{equation}
and such that
\begin{equation} \label{eq:sublin_distance}
\begin{split}
&\|u_h^k - u_h\|_{\infty,\infty} + \|m_h^k - m_h\|_{\infty,1}  \leq  \frac{\bar{C}_{\theta}}{\sqrt{kh^{d/2}}},\\
& \|u_h^k - u_h^{*} \|_{\infty,\infty} +  \|m_h^k - m_h^*\|_{\infty,1}  \leq  \bar{C}_{\theta}\left(h^{r} + \frac{1}{\sqrt{kh^{d/2}}} \right),
\end{split}
\end{equation}
for any $k \geq 1$. In particular, for $k \geq h^{-(2r + d/2)}$, we have
\begin{equation}
\|u_h^k - u_h^{*} \|_{\infty,\infty} +  \| m_h^k - m_h^* \|_{\infty,1} \leq  2\bar{C}_{\theta}h^r.
\end{equation}
\end{thm}

\begin{rem}
Let us emphasize that only the estimate \eqref{eq:sublin_rate} is mesh-independent. The estimates provided in \eqref{eq:sublin_distance} get worse as $h \rightarrow 0$.
\end{rem}

The proof of Theorem \ref{thm:main2} is given in Subsec.\@ \ref{subsec:proof_sublin}. It is based on Proposition \ref{prop:convergence1}, more precisely on estimates of the three fundamental constants $C_1$, $C_2$, and $C_3$. We will derive from these estimates some estimates of the constants $D_1$, $D_2$, and $c$ (as defined in  \eqref{eq:C}). A key point is that the estimate of the constant $D_1$ is mesh-independent. Our analysis mainly relies on a stability result for the discrete Fokker-Planck equation, for the $\ell^2$-norm.

In order to establish mesh-independent estimates of $D_2$ and $c$, so as to have a mesh-independent linear rate of convergence, we need to establish an $\ell^\infty$-estimate for the discrete Fokker-Planck, assuming that the involved vector field $v$ derives from a value function $u$. In the continuous case, such an estimate is relatively easy to deduce from the semi-concavity of $u$. The transposition of this analysis to a discrete setting is quite delicate; our analysis is restricted to the case of a separable running cost (with respect to the control variables).

A function $l\colon \mathbb{R}^d\rightarrow \mathbb{R}$ is said to be semi-concave if $l(x)-L\|x\|^2/2$ is concave, for some $L\geq 0$. This definition makes no sense when $l$ is a function defined on torus. In fact, one can check that a periodical function is concave if and only if it is  constant. Besides the previous definition, a second one based on a quadratic inequality is introduced in \cite[Def.\@ 1.1.1]{cannarsa2004} for functions on an open set. We use the latter one for our definition of semi-concavity on a torus and its discretization.

\begin{defn} \label{def:semi-concave1}[Semi-concave functions on the torus]
Let $L$ be a positive constant. A function $l^c\colon \mathbb{T}^d\rightarrow \mathbb{R}$ is said to be $L$-semi-concave if 
\begin{equation}\label{eq:semi-concave}
 l^c\left(x\right) \geq \frac{l^c(x+y)+l^c(x-y)}{2} - L\|y\|^2, \qquad \forall x,\,y \in \mathbb{T}^d.
\end{equation}
\end{defn}

\begin{rem}
If a function  $l^c\colon \mathbb{T}^d\rightarrow \mathbb{R}$ is $\mathcal{C}^2$, then it is semi-concave, as can be easily verified.
\end{rem}

We consider the following assumption.

\begin{ass}\label{ass:semi-concave-l} 
\begin{enumerate}
\item \emph{Semi-concavity}. There exists $L^c>0$, such that for any $(t,v,m)\in [0,1]\times \mathbb{R}^d\times \mathcal{D}$ with $\|v\|\leq M$, the functions $\ell^c(t,x,v)$, $f^c(t,x,m)$, and $g^c(x)$ are $L^c$-semi-concave with respect to $x$.
\item \emph{Separability}. There exist functions $\ell^c_i \colon  Q \times \mathbb{R} \rightarrow \mathbb{R}$ for $i=1,2,\ldots,d$, such that for any $(t,x,v)\in Q\times \mathbb{R}^d$,
\begin{equation*}
\ell^c(t,x,v) = \sum_{i=1}^d \ell^c_i(t,x,v_i),
\end{equation*}
where $v_i$ is the $i$-th coordinate of $v$.
\end{enumerate}
\end{ass}

In the next theorem we establish a linear and mesh-independent rate of convergence for the GFW algorithm with the line-search rule \eqref{eq:lambda}.

\begin{thm}[Linear rate]\label{thm:main3}
Let Assumptions \ref{ass:continuous}-\ref{ass:semi-concave-l} and condition \eqref{cond:CFL} hold true. In Algorithm \ref{alg2}, set $\lambda_k$ with the rule \eqref{eq:lambda}, for all $k \geq 1$. Then there exist three constants $c_{\theta} \in (0,1)$, $C_{\theta}>0$ and $\bar{C}_{\theta}$, independent of $\Delta t$ and $h$, such that
\begin{equation} \label{eq:linear_rate_theta}
\gamma_{k}  \leq C_{\theta} c_{\theta}^{k}, \quad \forall k \geq 1,
\end{equation}
and such that
\begin{equation} \label{eq:linear_stab_theta}
\begin{split}
& \|u_h^k - u_h\|_{\infty,\infty} + \|m_h^k - m_h\|_{\infty,1}  \leq \bar{C}_{\theta} c_{\theta}^{k/2}, \\
& \|u_h^k - u^{*}_h\|_{\infty,\infty} +  \|m_h^k - m^{*}_h\|_{\infty,1}  \leq  \bar{C}_{\theta}\left(h^{r} + c_{\theta}^{k/2} \right),
\end{split}
\end{equation}
for any $k \geq 1$. In particular, taking $k \geq 2r \log(h)/\log(c_{\theta})$, we have
\begin{equation}
\|u_h^k - u^{*}_h\|_{\infty,\infty} +  \|m_h^k - m^{*}_h\|_{\infty,1} \leq  2\bar{C}_{\theta}h^r.
\end{equation}
\end{thm}

Note that both estimates \eqref{eq:linear_rate_theta} and \eqref{eq:linear_stab_theta} are now mesh-independent. The proof of Theorem \ref{thm:main3} is given in Subsec.\@ \ref{subsec:proof_lin}.
 
\subsection{Proof of the sublinear rate of convergence} 
\label{subsec:proof_sublin}

We prove in this section Theorem \ref{thm:main2}.
Assumptions \ref{ass:continuous}-\ref{ass:potential} are supposed to be satisfied all along the subsection, as well as the condition \eqref{cond:CFL}. Our analysis relies on an energy estimate, obtained in Lemma \ref{lm:energy1}, which allows us to find first estimates of the convergence constants $C_1$, $C_2$, $C_3$, $D_1$, and $D_2$ (in Lemma \ref{lm:constant1}).

We define the forward discrete gradient as follows:
\begin{equation*}
\nabla_h^{+} \nu(x) = \Big( \frac{\nu(x+he_i)- \nu(x)}{h} \Big)_{i=1}^d, \qquad \forall \nu\in \mathbb{R}(S), \; \forall x\in S.
\end{equation*}

\begin{lem}[Energy estimate]\label{lm:energy1}
Let $\mu\in \mathbb{R}(\tilde{\mathcal{T}}\times S)$ satisfy the following equation:
\begin{equation*}
\begin{cases}
\ \left(Id - \theta \sigma\Delta t \Delta_h\right) \mu(t+1)  = & \big(Id+ (1-\theta )\sigma \Delta t \Delta_h  \big) \mu(t) - \Delta t \, \text{\emph{div}}_h \big(v (t)\mu(t) \big)    - \Delta t \, \text{\emph{div}}_h \big(\delta_v (t) \big), \\ 
\ \mu(0)   = \mu_0 ,
\end{cases} \\
\end{equation*}
where $\delta_v(t) \in \mathbb{R}^d(S)$ and $\|v\|_{\infty,\infty}\leq M$. Then,
\begin{equation}\label{eq:energy}
\max_{t\in\tilde{\mathcal{T}}} \big\|\mu(t)\big\|_2^2  \leq c(\sigma,\theta,M)\left(  \big\|\mu_0 \big\|_2^2 +  (1-\theta) \sigma \Delta t \big\|\nabla^{+}_h \mu_0 \big\|_2^2 + \frac{1}{\sigma(2\theta -1)} \sum_{\tau\in\mathcal{T}} \Delta t    \big\| \delta_v(\tau)\big\|^2_2 \right),
\end{equation}
where 
\begin{equation}\label{eq:c_energy}
c(\sigma,\theta,M) \coloneqq  1+\frac{M^2}{\sigma(2\theta -1)} \exp\left( \frac{M^2}{\sigma(2\theta-1)} \right).
\end{equation}
\end{lem}

\begin{proof}
We deduce from the proof of \cite[Prop.\@ 4.5]{BLP2022} that
\begin{equation}\label{eq:dmu3}
\begin{split}
&\frac{1}{2}\Big(\|\mu(t+1)\|_2^2 - \|\mu(t)\|_2^2 \Big)  + \theta \sigma \Delta t \big\|\nabla^{+}_h \mu(t+1) \big\|_2^2 \\
& \qquad \leq  - (1-\theta) \sigma \Delta t\big\langle \nabla^{+}_h \mu(t+1) , \nabla^{+}_h \mu(t)\big\rangle  + \Delta t \big( \gamma_1 + \gamma_2  \big),
\end{split}
\end{equation}
where 
\begin{equation*}
\gamma_1   = \sum_{x\in S}  \big\langle  \nabla_h \mu(t+1,x)  , \mu v(t,x)\big\rangle ,\qquad
\gamma_2   = \sum_{x\in S}  \big\langle  \nabla_h \mu(t+1,x)  ,  \delta_v(t,x)\big\rangle.
\end{equation*}
Let $\alpha_1 = \sigma (2\theta-1) >0$. Applying Young's inequality to the right-hand side of \eqref{eq:dmu3}, we have the following inequalities:
\begin{align*}
 - \big\langle \nabla^{+}_h \mu(t+1) , \nabla^{+}_h \mu(t)\big\rangle & \leq \frac{1}{2} \big\|\nabla^{+}_h \mu(t+1) \big\|_2^2 + \frac{1}{2} \big\|\nabla^{+}_h \mu(t) \big\|_2^2 ;\\
      \gamma_1  & \leq \frac{\alpha_1}{2} \big\|\nabla_h{\mu}(t+1) \big\|^2_2 +  \frac{1}{2\alpha_1} \big\| \mu {v}(t) \big\|^2_2 \leq \frac{\alpha_1}{2} \big\|\nabla_h ^{+}{\mu }(t+1) \big\|^2_2 +  \frac{M^2}{2\alpha_1}  \big\|\mu (t) \big\|^2_2;\\
       \gamma_2
& \leq \frac{\alpha_1}{2}\big\|\nabla^{+}_h{\mu}(t+1)\big\|^2_2  +  \frac{1}{2\alpha_1}  \big\|\delta_v(t)\big\|^2_2.
\end{align*}
Combining the above inequalities and \eqref{eq:dmu3}, it follows that
\begin{equation*}
\|\mu(t+1)\|_2^2 - \|\mu(t)\|_2^2 + (1-\theta) \sigma \Delta t \Big( \big\| \nabla^{+}_h \mu(t+1) \big\|_2^2 - \big\| \nabla^{+}_h \mu(t) \big\|_2^2 \Big) \\
\leq \Delta t \Big( \frac{M^2}{\alpha_1} \big\| \mu (t) \big\|^2_2 + \frac{1}{\alpha_1} \big\| \delta_v(t) \big\|^2_2 \Big).
\end{equation*}
Summing the previous inequality over $t$, it follows that for any $t\in\mathcal{T}$,
\begin{equation*}
\|\mu(t+1)\|_2^2 
\leq   \gamma +  \frac{M^2 \Delta t}{\alpha_1} \sum_{\tau=0}^t  \big\|\mu (\tau) \big\|^2_2,
\end{equation*}
where
\begin{equation*}
\gamma =   \|\mu_0\|_2^2 +  (1-\theta) \sigma  \Delta t\|\nabla^{+}_h \mu_0 \|_2^2 + \frac{1}{\alpha_1}\sum_{\tau\in\mathcal{T}}\Delta t \big\| \delta_v(\tau) \big\|^2_2.
\end{equation*}
We deduce from the discrete Gronwall inequality \cite{clark1987} that
\begin{align*}
\max_{t\in\tilde{\mathcal{T}}} \big\|\mu(t) \big\|_2^2 & \leq \gamma + T \gamma \frac{M^2\Delta t}{\alpha_1} \exp\left( \sum_{\tau\in\mathcal{T}}\frac{M^2\Delta t}{\alpha_1} \right)  = \gamma \left( 1+ \frac{M^2}{\alpha_1} \exp\left(\frac{M^2}{\alpha_1}\right) \right).
\end{align*}
The conclusion follows.
\end{proof}

We define next two constants $E_1$ and $E_3$, both independent of $\Delta t$ and $h$:
\begin{equation*}
E_1 =  c(\sigma,\theta,M) \Big( \|m_0^c\|_{\mathbb{L}^{\infty}}^2 + \frac{1}{2}\|\nabla m_0^c\|^2_{\mathbb{L}^{\infty}} \Big)
\quad \text{and} \quad E_3 = \frac{c(\sigma,\theta,M)}{\sigma(2\theta -1)}.
\end{equation*}
A constant $E_2$ will be introduced later on.

\begin{lem}\label{lm:constant1}
The constants $C_1$, $C_2$, and $C_3$ (as defined in \eqref{eq:def_C_12}-\eqref{eq:energy2}) satisfy the following inequalities:
\begin{equation}\label{eq:constants}
C_1  \leq  E_1 h^d , \qquad C_2 \leq \sqrt{C_1} \leq E_1^{1/2} h^{d/2}, \qquad 
 C_3  \leq E_3.
\end{equation}
As a consequence, for the constants $D_1$ and $D_2$ defined in \eqref{eq:C}, we have
\begin{equation}\label{CC1}
       D_1 \leq E_1 L_f^c, \quad  
    D_2 \leq (1+L_f^c) \sqrt{\frac{2E_1^{3/2}E_3 L_f^c}{\alpha^c}} h^{-d/4}.
\end{equation}
\end{lem}

\begin{proof}
The condition \eqref{cond:CFL} implies that
\begin{equation}\label{eq:cfl}
d (1-\theta) \sigma \Delta t \leq \frac{h^2}{2} \leq \frac{1}{2}.
\end{equation}
We let the reader verify that
\begin{equation*}
 \|m_0\|_2^2 + (1-\theta)\sigma \Delta t  \|\nabla^{+}_h m_0 \|_2^2  \leq \left( \|m_0^c\|_{\mathbb{L}^{\infty}}^2 + d (1-\theta) \sigma \Delta t \|\nabla m_0^c\|^2_{\mathbb{L}^{\infty}}\right) h^d.
\end{equation*}
Combining the above two estimates, we deduce that
\begin{equation*}
 \|m_0\|_2^2 + (1-\theta)\sigma \Delta t  \|\nabla^{+}_h m_0 \|_2^2  \leq \Big( \|m_0^c\|_{\mathbb{L}^{\infty}}^2 + \frac{1}{2} \|\nabla m_0^c\|^2_{\mathbb{L}^{\infty}} \Big) h^d.
\end{equation*}
Using this inequality in Lemma \ref{lm:energy1}, applied with $\delta_v(t)=0$ and $\mu_0 = m_0$, we obtain the estimate of $C_1$. Then, the estimate of $C_2$ is deduced from H\"older's inequality. The estimate of $C_3$ follows from equality \eqref{eq:delta_m} and Lemma \ref{lm:energy1} by taking $\mu_0=0$ and $\delta_v = (v_1-v_2)m_1$. Finally, \eqref{CC1} follows from \eqref{eq:C} and the previous estimates.
\end{proof}

\begin{proof}[Proof of Theorem \ref{thm:main2}]
Inequality \eqref{eq:sublin_rate} directly follows from Proposition \ref{prop:convergence1} and from the estimate $D_1\leq E_1 L_f^c$, with $C_{\theta} = 8 E_1L_f^c$. Then, using Lemma \ref{lm:u_m} and the estimate of $D_2$ obtained in \eqref{CC1}, we deduce that
\begin{equation*}
\|u^k - u_h\|_{\infty,\infty} + \|m^k - m_h\|_{\infty,2} \leq D_2 \sqrt{\gamma_k} \leq 4 E_1 L_f^c (1+L_f^c) \sqrt{\frac{E_1E_3}{\alpha^c}} \frac{1}{\sqrt{k h^{d/2}}}.
\end{equation*}
Recall that $C^*$ was introduced in Theorem \ref{thm:main1}.
The inequalities in \eqref{eq:sublin_distance} hold true with
\begin{equation*}
\bar{C}_{\theta} = \max \left\{ 4 E_1 L_f^c (1+L_f^c) \sqrt{\frac{E_1E_3}{\alpha^c}} \, , \, C^{*}\right\}.
\end{equation*}
The theorem is proved.
\end{proof}
 
\subsection{Proof of the linear rate of convergence} 
\label{subsec:proof_lin}

We prove in this subsection Theorem \ref{thm:main3}.
Assumptions \ref{ass:continuous}-\ref{ass:semi-concave-l} are supposed to be satisfied all along the subsection, as well as the condition \eqref{cond:CFL}. 
At a technical level, we look for a more precise estimate of the constant $C_2$ (see Lemma \ref{lm:constant2}), using an $\ell^{\infty}$-stability result for the map $\fpt\circ \vt \circ \hjbt$ (see Lemma \ref{lm:maximum}), whose proof is inspired from the continuous case (see for example \cite[Lem.\@ 5.3]{cardaliaguet2018}). Our analysis begins with a series of technical lemmas which will allow us to establish the semi-concavity of the value function.

\begin{defn}\label{def:semi-concave2}[Semi-concave functions on $S$] A function $l\colon S\rightarrow \mathbb{R}$ is said to be $L$-semi-concave if 
\begin{equation}\label{eq:semi-concave_discrete}
l\left(x\right) \geq \frac{l(x+y)+l(x-y)}{2} - L\|y\|^2, \qquad \forall x,\,y \in S.
\end{equation}
\end{defn}

\begin{lem}\label{lm:semi-concavity}
Let $L$, $L_1$, and $L_2$ be positive constants. The following statements hold:
\begin{enumerate}
\item Let $l_1 \colon S\rightarrow \mathbb{R}$ be $L_1$-semi-concave and $l_2 \colon S\rightarrow \mathbb{R}$ be $L_2$-semi-concave. For any $\lambda_1,\lambda_2 \geq 0$, the function $ \lambda_1 l_{1} + \lambda_2 l_{2}$ is $(\lambda_1L_1+\lambda_2 L_2)$-semi-concave.
\item  Let $\{l_{\omega} \colon  S\rightarrow \mathbb{R} \}_{\omega \in\Omega}$ be a family of $L$-semi-concave functions. Let $l \colon S\rightarrow \mathbb{R},\, l(x)= \inf_{\omega\in\Omega} l_{\omega} (x)$. Suppose that for all $x \in S$, it holds that $l(x_0)> - \infty$.  Then $l$ is $L$-semi-concave.
\end{enumerate}
\end{lem}

\begin{proof}
The first point is obtained by the definition \eqref{eq:semi-concave_discrete} and the non-negativity of $\lambda_1$ and $\lambda_2$. Let us prove the second point. Let $x_0 \in S$.
We have $l(x_0)>-\infty$. Then for any $\epsilon>0$, there exists $\omega_{\epsilon}\in \Omega$ such that
\begin{equation*}
l(x_0) \geq l_{\omega_{\epsilon}}(x_0) - \epsilon \geq \frac{l_{\omega_{\epsilon}}(x_0+y)+l_{\omega_{\epsilon}}(x_0-y)}{2} - L\|y\|^2 - \epsilon \geq \frac{l(x_0+y)+l(x_0-y)}{2} - L\|y\|^2 - \epsilon, 
\end{equation*}
where the second inequality follows from the semi-concavity of $l_{\omega_{\epsilon}}$. We deduce the $L$-semi-concavity at the point $x_0$ by the arbitrariness of $\epsilon$.
\end{proof}

\begin{lem} \label{lm:semi-concave1}
The functions $\ell$, $f$, and $g$, defined in \eqref{eq:grid}, are $L^c$-semi-concave with respect to $x$.
\end{lem}

\begin{proof}
The semi-concavity of $\ell$ and $g$ is a direct consequence of their definitions in \eqref{eq:grid}. Let us prove the semi-concavity of $f$. By the definition of $f$ in \eqref{eq:grid}, taking any $t \in \mathcal{T}$, $m\in \mathcal{P}(\mathcal{T},S)$, and $x,y \in S$, we have
\begin{equation*}
    \begin{split}
    &  f(t,x+y,m) + f(t, x-y, m)\\
       & \qquad = \frac{1}{h^d}\int_{z\in B_h(x+y)} f^c(t,z,\mathcal{R}_h(m)) dz + \frac{1}{h^d}\int_{z\in B_h(x-y)} f^c(t,z,\mathcal{R}_h(m)) dz \\ 
       & \qquad =  \frac{1}{h^d}\int_{z\in B_h(0)} f^c(t,x+y+z,\mathcal{R}_h(m)) + f^c(t,x-y+z,\mathcal{R}_h(m)) dz \\
       & \qquad \leq  \frac{1}{h^d}\int_{z\in B_h(0)} 2 f^c(t,x+z,\mathcal{R}_h(m)) + 2L^c\|y\|^2  dz \\
       &\qquad  =  \frac{1}{h^d} \int_{z\in B_h(x)} 2 f^c(t,z,\mathcal{R}_h(m)) dz + 2L^c\|y\|^2 = 2f(t,x,m) +2 L^c\|y\|^2. 
    \end{split}
\end{equation*}
The conclusion follows. 
\end{proof}

\begin{defn}
Let $A$ be a function from $\mathbb{R}(S)$ to $\mathbb{R}$. We say that $A$ is \textit{translation invariant} if for any $X\in\mathbb{R}(S) $ and any $y\in S$,
\begin{equation*}
    A(X(\cdot)) = A(X(\cdot - y)).
\end{equation*}
\end{defn}

\begin{lem}\label{lm:implicit0}
Suppose that $X, Y\in \mathbb{R}(S)$ satisfy the following equation for some $c>0$:
\begin{equation}\label{eq:implicit}
 (Id - c\Delta t \Delta_h) Y = X.
\end{equation} 
Let $A\colon \mathbb{R}(S)\rightarrow \mathbb{R}$ be a translation invariant function. We have the following statements:
\begin{itemize}
\item If $A$ is convex and l.s.c, then $A(Y)\leq A(X)$.
\item If $A$ is concave and u.s.c, then $A(Y)\geq A(X)$.
\end{itemize}
\end{lem}

\begin{proof} \label{pf:implicit}
For any $Z\in \mathbb{R}(S)$ and $y\in S$, we define $\tau^y Z \in \mathbb{R}(S)$ by
\begin{equation*}
\tau^y Z(\cdot) = Z(\cdot -y).
\end{equation*}
Let $\gamma = c\Delta t / h^2$.
 We define $\mathbb{S}_{X} \colon \mathbb{R}(S) \rightarrow \mathbb{R}(S)$,
\begin{equation*}
\mathbb{S}_{X}(\mu) = \frac{1}{1+2d\gamma} \left\{ X + \gamma\sum_{j=1}^d \left(\tau^{he_j}\mu + \tau^{-he_j}\mu \right)\right\}.
\end{equation*}
By the proof of \cite[Lemma 2.6]{BLP2022}, $\mathbb{S}_X$ is a contraction mapping and $Y$ is the fixed point of $\mathbb{S}_X$.
  
Suppose that $A$ is l.s.c.\@ and convex. Suppose that $\mu\in \mathbb{R}(S)$ satisfy $A(\mu) \leq A(X)$. By the convexity of $A$, we have
\begin{equation*}
\begin{split}
A(\mathbb{S}_{X}(\mu)) &\leq \frac{1}{1+2d\gamma} \left\{ A(X) + \gamma\sum_{j=1}^d \left(A(\tau^{he_j}\mu) + A(\tau^{-he_j}\mu) \right)\right\} \\
&= \frac{1}{1+2d\gamma} \left\{ A(X) + \gamma\sum_{j=1}^d \left(A(\mu) + A(\mu) \right)\right\} \leq A(X),
\end{split}
\end{equation*}
where the second line follows from the translation invariance of $A$. Therefore, $A(\mathbb{S}^k_X(X)) \leq A(X)$ for any $k\geq 1$. Since  $Y = \lim_{k\rightarrow \infty} \mathbb{S}^k_{X}(X)$, by the lower-semi-continuity of $A$, we have
\begin{equation*}
A(Y)\leq \liminf_{k\rightarrow \infty} A(\mathbb{S}^k_X(X)) \leq A(X).
\end{equation*}
For the case where $A$ is u.s.c and concave, it suffices to apply the previous result to $-A$.
\end{proof}

\begin{lem}\label{lm:implicit}
Let $X, Y\in \mathbb{R}(S)$ satisfy \eqref{eq:implicit}. Then, the following statements hold.
\begin{enumerate}
\item Maximum/minimum principle: 
\begin{equation*}
\min_{x\in S} X(x) \leq \min_{x\in S} Y(x) \leq \max_{x\in S} Y(x) \leq \max_{x\in S} X(s).
\end{equation*}
\item Conservation of the mass:
\begin{equation*}
\sum_{x\in S}Y(s) = \sum_{x\in S} X(s).
\end{equation*}
\item Conservation of the Lipschitz constant: If $X(x)$ is $L$-Lipschitz, then $Y(x)$ is $L$-Lipschitz.
\item Conservation of the semi-concavity constant: If $X(x)$ is $L$-semi-concave, then $Y(x)$ is $L$-semi-concave.
\end{enumerate} 
\end{lem}

\begin{proof}
We use Lemma \ref{lm:implicit0} for the proof. The key point is the choice of the translation invariant function $A$ in Lemma \ref{lm:implicit0}. Keep in mind that the maximum (resp.\@ minimum) of a family of linear functions is l.s.c. and convex (resp.\@ u.s.c. and concave) in finite dimensions.
    
For point (1), it suffices to take $A (X) = \min_{x\in S} \{X(x)\}$ and $A (X) = \max_{x\in S} \{X(x)\}$. For point (2), we take $A(X) = \sum_{x\in S}X(x)$. 
For point (3), we take 
\begin{equation*}
A(X) = \max_{x\in S} \max_{y\in S, y\neq 0}  \frac{X(x+y)-X(x)}{\|y\|}.
\end{equation*}
Finally, for point (4), we take 
\begin{equation*}
A(X) = \max_{x\in S}\max_{y\in S, y\neq 0} \frac{X(x+y) + X(x-y) -2 X(x)}{\|y\|^2}.
\end{equation*}
The conclusion follows.
\end{proof}

\begin{lem}[Semi-concavity of the value function] \label{lm:semi-concave}
Let Assumptions \ref{ass:continuous}, \ref{ass:semi-concave-l}(1)  and condition \eqref{cond:CFL} hold true. Then for any $m\in\mathcal{P}(\tilde{\mathcal{T}},S)$, the unction \textnormal{$u = \hjbt(m)$} is $3L^c$-semi-concave with respect to $x$.
\end{lem}

\begin{proof}
Observe that $\hjbt(m)$ is equivalent to the formulation below:
\begin{equation}\label{eq:dp_u}
\begin{cases}
\ \left(Id - \theta \sigma\Delta t \Delta_h\right) u(t+1/2)  =  u(t + 1) ; \\[0.5em]
\ u(t,x) =   \Delta t \inf_{\|\omega\|\leq M} \Big\{ f(t,x,m(t)) + \ell(t,x,\omega) + \big\langle  \omega, \nabla_h u(t+1/2,x)  \big\rangle\Big\} \\[0.5em]
\qquad \qquad +\big(Id+ (1-\theta )\sigma \Delta t \Delta_h  \big)u(t+1/2)(x), \qquad \forall x \in S; \\[0.5em]
\ u(T,x) =g(x), \qquad \forall x \in S.
\end{cases}
\end{equation}
We prove the lemma by induction. For $t=T$, by the terminal condition, it is obvious that
\begin{equation*}
      u\left(T,x\right) \geq \frac{ u(T,x+y)+ u(T,x-y)}{2} - L^c\|y\|^2, \qquad \forall x,\,y \in S.
\end{equation*}
Suppose that for some $t\in\mathcal{T}$, we have
\begin{equation*}
    u\left(t+1,x\right) \geq \frac{ u(t+1,x+y)+ u(t+1,x-y)}{2} - (2(T-1-t)\Delta t +1)L^c\|y\|^2, \qquad \forall x,\,y \in S.
\end{equation*}
Since $u(t+1)$ and $u(t+1/2) $ satisfy the implicit scheme \eqref{eq:implicit}, by Lemma \ref{lm:implicit}(4), we have
\begin{equation*}
    u\left(t+1/2,x\right) \geq \frac{ u(t+1/2,x+y)+ u(t+1/2,x-y)}{2} - (2(T-1-t)\Delta t +1)L^c\|y\|^2, \qquad \forall x,\,y \in S.
\end{equation*}

Let $r' = (1 - \theta) \sigma \Delta t/h^2$. The second equation in \eqref{eq:dp_u} can be written as follows:
\begin{equation*}
    u(t,x) = \inf_{\|\omega\|\leq M} l_{\omega}(t,x),
\end{equation*}
where
\begin{equation*}
\begin{split}
    l_{\omega}(t,x) \coloneqq & (1-2dr')u(t+1/2,x) +\sum_{i=1}^d \Big( r' +\frac{\omega_i}{2h}\Big) u(t+1/2,x+he_i)\\
    & +\sum_{i=1}^d \Big( r' -\frac{\omega_i}{2h}\Big) u(t+1/2,x-he_i) +  \Delta t (f(t,x,m(t)) + \ell(t,x,\omega)).
\end{split}
\end{equation*}
By condition \eqref{cond:CFL}, the coefficients of the above equation are positive for any $\|w\|\leq M$. Then by Lemma \ref{lm:semi-concavity}(1), the semi-concavity of $u(t+1/2,\cdot)$, $f$, and $\ell$,  we have that $l_{\omega}$ is $(2(T-t)\Delta t + 1)L^c$-semi-concave. Since $u(t,x)>-\infty$ for any $x\in S$, we deduce from Lemma \ref{lm:semi-concavity}(2) that
\begin{equation*}
     u(t,x) \geq \frac{u(t,x+y) + u(t,x-y)}{2} - (2(T-t)\Delta t +1)L^c\|y\|^2, \qquad \forall x,\,y \in S.
\end{equation*}
The conclusion follows by induction.
\end{proof}

We have the following regularity result for the discrete Hamiltonian $H$ (defined in \eqref{eq:discrete_hamiltonian}).

\begin{lem}\label{lm:lip}
Let Assumptions \ref{ass:continuous} and \ref{ass:semi-concave-l}(2) hold true. Then, for any $(t,x)\in \mathcal{T}\times S$ and $\|p\|\leq \sqrt{d}(L^c_{\ell}+L^c_{f}+L^{c}_{g})$,
\begin{equation}\label{eq:H_separate}
  H(t,x,p) = \sum_{i=1}^d H^c_i(t\Delta t,x,p_i), \qquad H_p(t,x,p) = \left( \frac{\partial H^c_i}{\partial p_i}(t\Delta t,x,p_i)\right)_{i=1}^d,
\end{equation}
where $p_i$ is the $i$-th coordinate of $p$ and
\begin{equation}\label{eq:Hc_i}
   H^c_i(t,x,p_i) = \sup_{v_i} - v_i p_i - \ell^c_i(t,x,v_i).
\end{equation}
Moreover, $ \frac{\partial H^c_i}{\partial p_i}(t\Delta t,x,p_i) $ is $L^c_{\ell}/{\alpha^c}$-Lipschitz with respect to $x$.
\end{lem}

\begin{proof}
Equality \eqref{eq:H_separate} is from \cite[Lemma 5.1]{BLP2022} and the separable form of $\ell^c$. The Lipschitz continuity of $ \frac{\partial H^c_i}{\partial p_i}$ is proved with the same argument as the one the proof of \cite[Lemma 2.7]{BLP2022}.
\end{proof}

\begin{lem}[Lipschitz continuity of the value function]\label{lm:lip1}
Let Assumption \ref{ass:continuous} and condition \eqref{cond:CFL} hold true. For any $m\in\mathcal{P}(\tilde{\mathcal{T}},S)$, let \textnormal{$u=\hjbt(m)$} and \textnormal{$v = \vt(u) $}. Then $u$ is $(L_{\ell}^c+L_{f}^c+L_g^c)$-Lipschitz with respect to $x$ and $\|v\|_{\infty,\infty}\leq M$.
\end{lem}
\begin{proof}
See \cite[Lemma 4.3]{BLP2022}.
\end{proof}

\begin{lem}[$\ell^{\infty}$-stability]\label{lm:maximum}
Let Assumptions \ref{ass:continuous}, \ref{ass:semi-concave-l} and condition \eqref{cond:CFL} hold true. Then,
\textnormal{
\begin{equation*}
   \sup_{\mu\in \mathcal{P}(\tilde{\mathcal{T}},S)} \|\fpt \circ \vt \circ \hjbt (\mu)\|_{\infty,\infty} \leq \exp\left(\frac{d(L_{\ell}^c + 6 L^c)}{\alpha^c }\right) \|m_0\|_{\infty}.
\end{equation*}}
\end{lem}

\begin{proof}
Let $\mu\in \mathcal{P}(\tilde{\mathcal{T}},S)$, let $u = \hjbt (\mu)$, let $v = \vt(u)$, and let $m = \fpt(v)$. 
Observe that $m = \fpt(v)$ is equivalent to the formulation below:
\begin{equation}\label{eq:fpt}
 \begin{cases}
 \         m(t+1/2) = \big(Id+ (1-\theta )\sigma \Delta t \Delta_h  \big) m(t)   - \Delta t \text{div}_h \big(v (t)m(t) \big) ; \\[0.4em]
 \          \left(Id - \theta \sigma\Delta t \Delta_h\right) m(t+1) = m(t+1/2); \\[0.4em]
  \         m(0) = m_0.
    \end{cases}
\end{equation}
Let us first compare $\|m(t+1/2,\cdot)\|_{\infty}$ and $\|m(t+1,\cdot)\|_{\infty}$. Since $m(t+1)$ and $m(t+1/2)$ satisfy the implicit scheme \eqref{eq:implicit}, by Lemma \ref{lm:implicit}(1), we have 
\begin{equation}\label{eq:maximum1}
\|m(t+1,\cdot)\|_{\infty}\leq\|m(t+1/2,\cdot)\|_{\infty}.
\end{equation}
Then, we compare $\|m(t,\cdot)\|_{\infty}$ and $\|m(t+1/2,\cdot)\|_{\infty}$. Let $r' = (1 - \theta) \sigma \Delta t/h^2$.  The first equation in \eqref{eq:fpt} shows that for any $(t,x) \in \mathcal{T}\times S$,
\begin{equation*}
\begin{split}
     m(t+1/2,x) = {}&(1-2d\gamma')m(t,x) + \sum_{i=1}^d \left(\gamma' - \Delta t\frac{v_i(t,x+he_i)}{2h}\right) m(t,x+he_i) \\
      & + \sum_{i=1}^d\left(\gamma' + \Delta t \frac{v_i(t,x-he_i)}{2h}\right) m(t,x - he_i).
\end{split}
\end{equation*}
Condition \eqref{cond:CFL} implies that all the coefficients in the above equation are positive. Therefore, for any $(t,x) \in \mathcal{T}\times S$,
\begin{equation}\label{eq:maximum2}
    m(t+1/2,x) \leq \left(1 - \Delta t \sum_{i=1}^d \frac{ v_i(t,x+he_i) - v_i(t,x-he_i)}{2h }\right)\|m(t,\cdot)\|_{\infty} = \left(1 - \Delta t \text{div}_h v\right)\|m(t,\cdot)\|_{\infty}.
\end{equation}
By Lemma \ref{lm:lip1} and Lemma \ref{lm:implicit}(3), we have $\|\nabla_h u(t+1/2,x)\|\leq \sqrt{d}(L^c_{\ell}+L^c_{f}+L^{c}_{g}) $ for any $(t,x) \in \mathcal{T}\times S$. Then, formula \eqref{eq:H_separate} implies that
\begin{equation}\label{eq:div_hv}
\begin{split}
    - \text{div}_h v(t,x)
     = {}&\frac{1}{2h} \sum_{i=1}^d \frac{\partial H^c_{i}}{\partial p_i} (t\Delta t,x+he_i, (\nabla_h u (t+1/2,x+he_i))_i) \\
   &  \qquad - \frac{\partial H^c_{i}}{\partial p_i} (t\Delta t,x-he_i, (\nabla_h u (t+1/2,x-he_i))_i) \\
    \leq {}& \frac{dL^c_{\ell}}{\alpha^c} + \frac{1}{2h} \sum_{i=1}^d \frac{\partial H^c_{i}}{\partial p_i} (t\Delta t,x, (\nabla_h u (t+1/2,x+he_i))_i) \\
 & \qquad  - \frac{\partial H^c_{i}}{\partial p_i} (t\Delta t,x, (\nabla_h u (t+1/2,x-he_i))_i),
\end{split}
\end{equation} 
where the last inequality follows from the Lipschitz-continuity of $\frac{\partial H^c_{i}}{\partial p_i} $ with respect to $x$ (established in Lemma \ref{lm:lip}). Since $H^c_i$ is convex on $p_i$, the derivative $\frac{\partial H^c_{i}}{\partial p_i} $ is non-decreasing with respect to $p_i$. Furthermore, we know that $\frac{\partial H^c_{i}}{\partial p_i} $ is $1/\alpha^c$-Lipschitz on $p_i$ by the strong convexity of $\ell^c$. It follows that for any $(t,x)\in Q$ and $p_i^1, p_i^2 \in \mathbb{R}$,
\begin{equation*}
    \frac{\partial H^c_{i}}{\partial p_i} (t,x,p_i^1) -\frac{\partial H^c_{i}}{\partial p_i} (t,x,p_i^2) \leq \max\left\{0, \frac{1}{\alpha^c} (p_i^1 - p_i^2)\right\}.
\end{equation*}
Applying the above inequality to \eqref{eq:div_hv}, we have 
    \begin{equation*}
\begin{split}
&   - \text{div}_h v(t,x)\leq \frac{dL^c_{\ell}}{\alpha^c} + \frac{1}{2\alpha^c h}\sum_{i=1}^d \max\left\{0, (\nabla_h u (t+1/2,x+he_i))_i - (\nabla_h u (t+1/2,x-he_i))_i\right\}\\
& \qquad   = \frac{dL^c_{\ell}}{\alpha^c} + \frac{1}{\alpha^c}\sum_{i=1}^d \max\left\{0,  \frac{u(t+1/2,x+2he_i) + u(t+1/2,x-2he_i) -2 u(t+1/2,x) }{4h^2 }\right\}.
\end{split}
\end{equation*}
By Lemma \ref{lm:semi-concave}, for any $(t,x,y)\in \mathcal{T}\times S^2$ and $y\neq 0$, we have
\begin{equation*}
   \frac{u(t+1/2,x+y) + u(t+1/2,x-y) -2 u(t+1/2,x) }{\|y\|^2 } \leq 6 L^c.
\end{equation*}
Taking $y=2h e_i$, it follows that
\begin{equation}\label{eq:maximum3}
 - \text{div}_h v(t,x) \leq  \frac{d(L_{\ell}^c + 6 L^c)}{\alpha^c }.
\end{equation}
Combining \eqref{eq:maximum1}, \eqref{eq:maximum2}, and \eqref{eq:maximum3}, we have \begin{equation*}
    \|m(t+1,\cdot)\|_{\infty} \leq \left(1 + \Delta t  \frac{d(L_{\ell}^c + 6 L^c)}{\alpha^c }\right) \|m(t,\cdot)\|_{\infty}.
\end{equation*}
Since $\Delta t = 1/T$, the conclusion follows.
\end{proof}

We are now ready to derive an improved estimate of $C_2$ (in comparison with the one in \eqref{eq:constants}). We define the constant $E_2$ as follows:
\begin{equation*}
    E_2 =  \exp\left(\frac{d(L_{\ell}^c + 6 L^c)}{\alpha^c }\right)\|m^c_0\|_{\mathbb{L}^{\infty}}.
\end{equation*}
It is independent of $\Delta t$ and $h$.

\begin{lem}\label{lm:constant2}
The constants $C_1$, $C_2$, and $C_3$ defined in \eqref{eq:def_C_12}-\eqref{eq:energy2} satisfy the following inequalities:
\begin{equation}\label{eq:C2_2}
C_1  \leq  E_1 h^d , \qquad C_2 \leq E_2 h^d, \qquad
C_3  \leq E_3.
\end{equation}
As a consequence, for the constants defined in \eqref{eq:C}, we have
\begin{equation}\label{CC2}
    \begin{split}
         D_1 \leq E_1 L_f^c, \quad  D_2 \leq (1+L_f^c) \sqrt{\frac{2E_1 E_2 E_3 L_f^c}{\alpha^c}}, \quad  c \leq c_{\theta} \coloneqq  \max\left\{1 - \frac{\alpha^c}{4E_2E_3L_f^c} \, , \, \frac{1}{2}\right\}.
    \end{split}
\end{equation}
\end{lem}

\begin{proof}
The estimates of $C_1$ and $C_3$ are the same as in \eqref{eq:constants}, and the estimate of $C_2$ is a direct consequence of Lemma \ref{lm:maximum} and the regularity of $m^c_0$. Then \eqref{CC2} is deduced from \eqref{eq:C2_2} and \eqref{eq:C}. 
\end{proof}

\begin{proof}[Proof of Theorem \ref{thm:main3}]
Inequality \eqref{eq:linear_rate_theta} holds true with $C_{\theta} = 4 E_1 L_f^c$, as a direct consequence of Proposition \ref{prop:convergence1}.
Inequality \eqref{eq:linear_stab_theta} is established in similar fashion to inequality \eqref{eq:sublin_distance}. It holds true with
\begin{equation*}
    \bar{C}_{\theta} = \max\left\{ 2E_1L_f^c (1+L_f^c) \sqrt{\frac{2 E_2 E_3}{\alpha^c}} \,  ,  \, C^{*} \right\}.
\end{equation*}
The theorem is proved.
\end{proof}

\subsection{Discussion on convergence constants} 

In this subsection, we study the dependence of the convergence constants $C_{\theta}$ and $c_{\theta}$ (appearing in Theorem \ref{thm:main3}) with respect to the viscosity parameter $\sigma$ and the Lipschitz constant $L_f^c$ of the coupling term $f^c$.
First, let us recall the constant $c(\sigma, \theta, M)$, introduced in \eqref{eq:c_energy},
\begin{equation*}
    c(\sigma,\theta,M) =  1+\frac{M^2}{\sigma(2\theta -1)} \exp\left( \frac{M^2}{\sigma(2\theta-1)} \right).
\end{equation*}
It is not difficult to see that $c(\sigma, \theta, M)$ decreases and converges to $1$ as $\sigma$ goes to $+\infty$. The constant $E_2$ is independent of $\sigma$ and $L_f^c$ by its definition (assuming that the change of $L_f^c$ has no impact on the semi-concavity constant of $f^c$).

By the proofs in the previous subsection, we can give the following explicit formulas of $C_{\theta}$ and $c_{\theta}$ in Theorem \ref{thm:main3} (without using $E_1$ and $E_3$):
\begin{equation}\label{eq:C_theta}
    C_{\theta}  = {} 4 c(\sigma,\theta,M)  L_f^c \left( \|m_0^c\|_{\mathbb{L}^{\infty}}^2 + \frac{1}{2}\|\nabla m_0^c\|^2_{\mathbb{L}^{\infty}}\right), \qquad  c_{\theta} = \max \left\{ 1 -  \frac{\sigma\alpha^c (2\theta -1)}{4 c(\sigma,\theta,M)L_f^c E_2}\, , \, \frac{1}{2} \right\}.
\end{equation}
 
\begin{lem} \label{lm:linear_conv}
For the constants in \eqref{eq:C_theta}, we have the following.
\begin{enumerate}
\item Fix $L_f^c$ and $\theta$. There exists $\sigma^{*} > 0$ and $C^{*}_1>0$, such that for any $\sigma \geq \sigma^{*}$, we have
\begin{equation*}
C_{\theta}  \leq C^{*}_1 \left( \|m_0^c\|_{\mathbb{L}^{\infty}}^2 + \frac{1}{2}\|\nabla m_0^c\|^2_{\mathbb{L}^{\infty}}\right),
\qquad  c_{\theta} = \frac{1}{2}.
\end{equation*}
\item Fix $\sigma$ and $\theta$. There exists $L^{*} > 0$ and $C^{*}_2 >0$, such that for any $L^c_f \leq L^{*}$, we have
\begin{equation*}
C_{\theta}  \leq  C^{*}_2\left( \|m_0^c\|_{\mathbb{L}^{\infty}}^2 + \frac{1}{2}\|\nabla m_0^c\|^2_{\mathbb{L}^{\infty}}\right),
\qquad  c_{\theta} = \frac{1}{2}. 
\end{equation*}
\end{enumerate}
\end{lem}

\begin{proof}
Point (1) follows from the monotonicity of $c(\sigma, \theta, M)$ w.r.t.\@ $\sigma$ and the fact that $M$ is independent of $\sigma$. We can take $C_1^{*} = 5 L_f^c $ for example.  To prove (2), we first notice that if $L_c^c \leq 1$, then $M\leq M^{*}$, where $M^{*}$ is defined by \eqref{eq:cons_M}, replacing $L_f^c$ with $1$. The monotonicity of $c(\sigma, \theta, M)$ w.r.t.\@ $M$ shows that $c(\sigma, \theta, M) \leq c(\sigma, \theta, M^{*})$. Since $c(\sigma, \theta, M^{*})L_f^c$ goes to $0$ as $L_f^c$ goes to $0$, we prove the existence of $L^{*}$, and $C_2^{*} = 4 c(\sigma, \theta, M^{*}) L^{*}$.
\end{proof}

From the proof of Proposition \ref{prop:convergence1}, we know that $c_{\theta} = 1/2$ implies that $\lambda_k = 1$ for any $k\geq 0$. In other words, Algorithm \ref{alg2} is equivalent to the so-called best-response iteration, i.e., for any $k\geq 0$,
\begin{equation*}
(m^{k+1}, w^{k+1}) = \brt (m^k).
\end{equation*}
Combined with Lemma \ref{lm:linear_conv}, we have the following observations.
\begin{enumerate}
\item \textit{High-viscosity} case: let $L_f^c$ be fixed, if $\sigma$ is large enough, then the best response iteration has a linear convergence rate with a factor $1/2$.
\item \textit{Weak-coupling} case: let $\sigma$ be fixed,  if $L_f^c$ is small enough, then the best response iteration has a linear convergence rate with a factor $1/2$.
\end{enumerate}

\section{Numerical tests}

\subsection{Problem formulation}

In this section, we consider an example of \eqref{eq:mfg} in dimension one. We identify the torus with the segment $[0,1]$. The initial distribution is concentrated around the point $0.5$, the running cost is a quadratic function of the control, and the terminal cost $g(x)$ decreases to $0$ as $x$ goes to zero. Additionally, we consider a non-local congestion term which penalizes the density of the agents within the intervals $[0.2,0.3]$ and $[0.7,0.8]$. We will refer to $[0.2,0.3]\cup[0.7,0.8]$ as the
\emph{congestion-sensitive zone}.

To model this situation, let us introduce the functions $\varphi_{A,k} \in \mathcal{C}^{\infty}(\R)$ and $\phi_{A,k,l_1,l_2} \in \mathcal{C}^{\infty}(\R)$, parameterized by $ A>0$, $k>0$, $0<l_1 < l_2 <1$ and defined by
\begin{equation*}
    \varphi_{A,k} (x) =
    \begin{cases}
        \ A e^{-\frac{1}{1-(kx)^2}}, \qquad &\text{if } |x| < \frac{1}{k},\\
        \ 0 & \text{otherwise},
    \end{cases} \qquad  \phi_{A,k,l_1,l_2} (x)= 
    \begin{cases}
        \ \varphi_{A,k} (x-l_1) , \qquad &\text{if } x < l_1, \\
        \ Ae^{-1} &\text{if } l_1\leq x \leq l_2, \\
        \ \varphi_{A,k} (x-l_2), &\text{otherwise}.
    \end{cases}
\end{equation*}
The function $\varphi_{A,k}$ is a smooth approximation of the piecewise constant function equal to $Ae^{-1}$ on $[x-1/k,x+1/k]$ and zero elsewhere. The function $\phi_{A,k,l_1,l_2}$ is a smooth approximation of the piecewise constant function equal to $Ae^{-1}$ on $[l_1-1/k,l_2+1/k]$ and zero elsewhere.

The data of our one-dimensional MFG is 
parameterized by five positive numbers $a_1$, $a_2$, $k_0$, $k_1$, and $k_2$ and defined by: For any $(t,x)\in Q$, any $v\in \R$, and any $m\in \mathcal{D}(\mathbb{T})$,
\begin{itemize}
    \item $\ell^c(t,x,v) = \frac{1}{2} v^2$;
    \item $m^c_0(x) = \phi_{1,k_0, 0.49, 0.51}(x)/ \| \phi_{1,k_0, 0.49, 0.51}\|_{\mathbb{L}^1}$ ;
    \item $g^c (x)= \phi_{a_1,k_1,1/k_1, 1-1/k_1}(x)$;
    \item $f^c(t,x,m) = h^c(x) \int_{0}^1h^c(y)m(y) dy$, where
    $h^c(x) =  \phi_{a_2,k_2, 0.24, 0.25}(x) + \phi_{a_2,k_2, 0.75, 0.76}(x)$.
\end{itemize}
We take $a_1=2$, $a_2 = 20$, $ k_0 = 10$, $k_1 = 3$, and $k_2 = 20$. The functions $m_0^c$, $g^c$, and $h^c$ are shown in Figure \ref{fig_data}. Moreover, we fix the viscosity coefficient $\sigma = 0.02$.
In this context, the agents have their initial condition around 0.5. The terminal cost $g^c$ is an incentive to move to the point 0 (which is the same as the point 1, since we identify the torus with $[0,1]$). If there was no congestion cost and no diffusion coefficient, the agents would move at a constant speed because of the quadratic running cost. The congestion cost $f^c$ is an incentive to spend less time in the congestion-sensitive area.

\begin{figure}[htbp]
	\centering
	\includegraphics[width=0.6\linewidth]{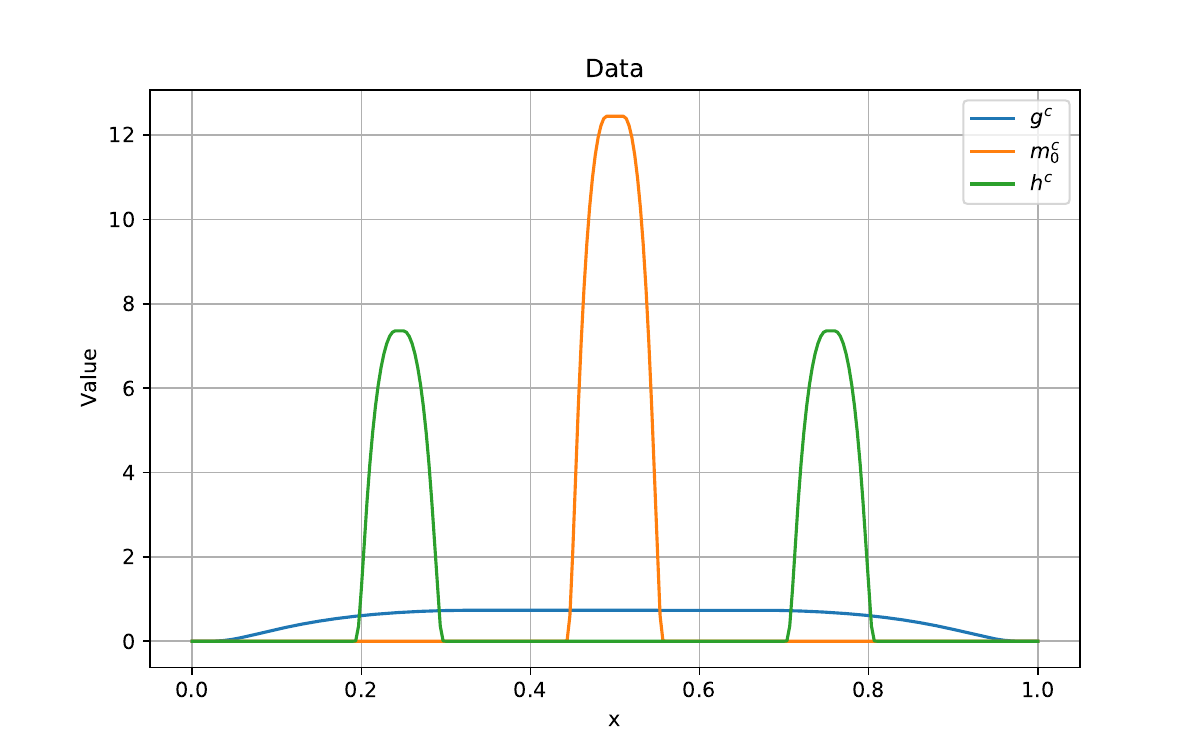}
	\caption{Data of the one-dimensional MFG with $a_1=2$, $a_2 = 20$, $ k_0 = 10$, $k_1 = 3$, and $k_2 = 20$.}
	\label{fig_data}
\end{figure}

We can verify that this one-dimensional MFG satisfies Assumptions \ref{ass:continuous}, \ref{ass:potential}, \ref{ass:semi-concave-l}, with the constants in Assumption \ref{ass:continuous} satisfying
\begin{equation*}
    \alpha^c =1,\quad L_{\ell}^c =0, \quad L_g^c \leq a_1k_1, \quad L_f^c \leq  \frac{a_2^2 k_2}{e}. 
\end{equation*}
Furthermore, \cite[Assumption C, Appx.\@ B]{BLP2022} holds true for this example, which implies Assumption \ref{ass:sol+} for any $r<1$ by \cite[Thm.\@ B.2]{BLP2022}. 

\subsection{Results}

For the discretization of the system, we first choose the parameters
\begin{equation*}
\theta = 0.8,   \quad h=1/300 ,\quad \text{and} \quad \Delta t = \frac{h^2}{2(1-\theta)\sigma} = 1/720,
\end{equation*}
and we present the outcome of Algorithm \eqref{eq:theta_mfg} after $1000$ iterations of the GFW algorithm, for step-sizes determined by line-search. For a better interpretation of the result, we also present the solution of the problem obtained by removing the congestion term $f^c$, which is a simple stochastic optimal control problem that can be solved in one iteration of the GFW method.

We present the equilibrium distribution of the agents in Figure \ref{fig_distribution} (without congestion term on the left, with congestion on the right). Note that the vertical axis corresponds to the time variable and is oriented downwards. We also present the restriction of the equilibrium distribution to the time interval $[0.2,0.8]$ in Figure \ref{fig_distribution_s}, with another color scale.
As the time progresses, the agents are transported towards the target points $0$ and $1$. The congestion term leads to a reduced density in the congestion-sensitive zone: We see two dark blue vertical areas corresponding to this zone.
We also see that at time $t \approx 0.35$, a significant part of the agents is still located around $0.5$ and has not crossed yet the sensitive zone, in comparison with the case without $f^c$.
Similarly, we present the optimal control $v$ in Figure \ref{fig_oc_u} (without congestion term on the left, with congestion term on the right). Unsurprisingly, the agents must have a high velocity (in absolute value) in the sensitive zone. It is interesting to see that for $t$ close to zero and for the agents not that close to $0.5$, there is an incentive to ``rush" to the sensitive zone. Finally, we display the value functions for the two problems in Figure \ref{fig_oc_u}. In the present setting, note that the optimal control is the discrete gradient of the value function.



\begin{figure}[htbp]
    \centering
    \begin{subfigure}{1\textwidth}
        \centering
        \includegraphics[width=\linewidth]{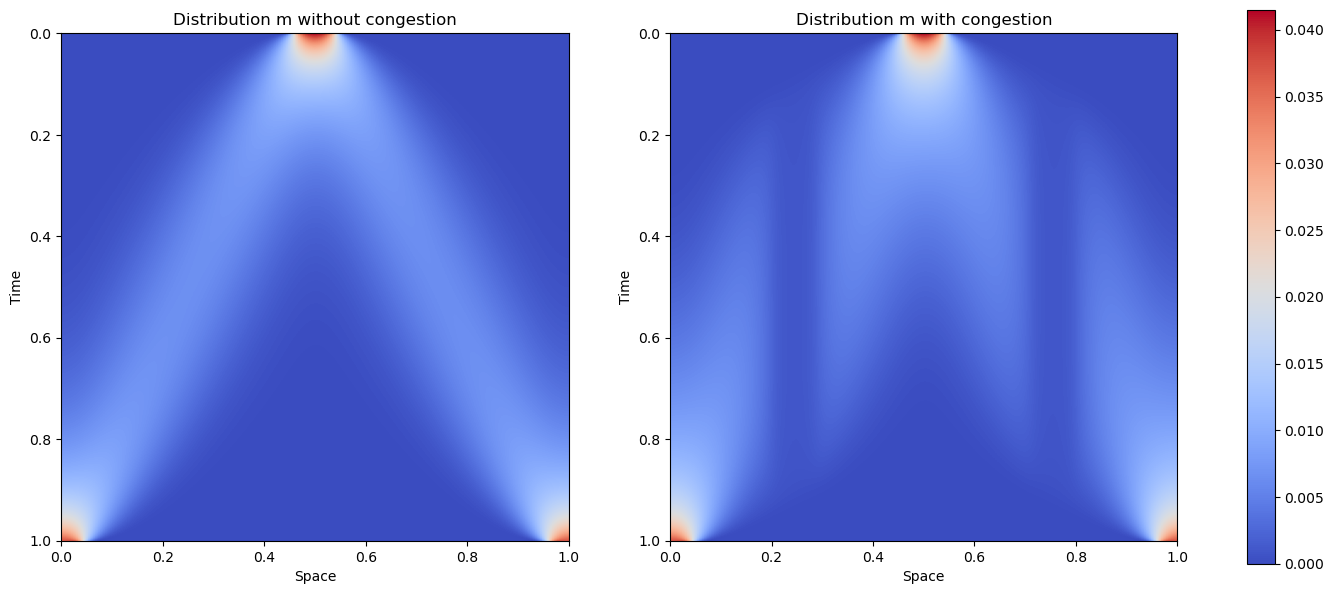}
        \caption{Comparison of distributions in the time horizon $[0,1]$: the case without $f^c$ (left), the case with $f^c$ (right).}
        \label{fig_distribution}
    \end{subfigure}
    \hfill

\vspace{0cm}
    
    \begin{subfigure}{1\textwidth}
        \centering
        \includegraphics[width=\linewidth]{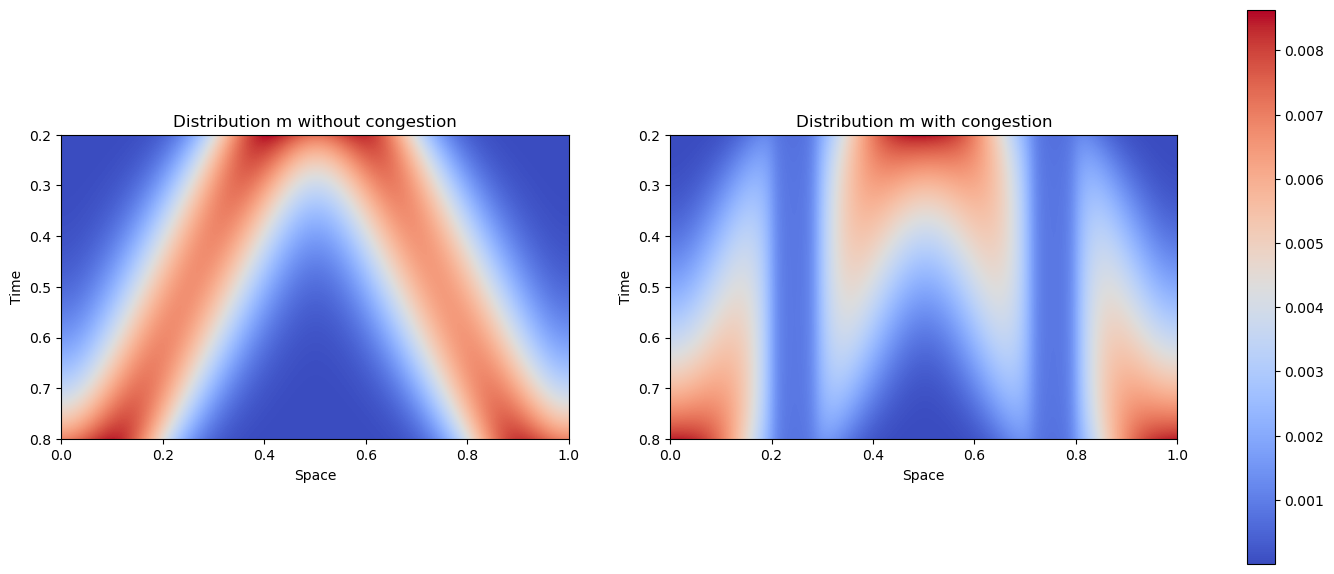}
        \caption{Comparison of distributions in the time horizon $[0.2,0.8]$: the case without $f^c$ (left), the case with $f^c$ (right).}
        \label{fig_distribution_s}
    \end{subfigure}
    \medskip
    \caption{Distributions}
    \label{fig_m}
\end{figure}

\begin{figure}[htbp]
    \centering
    \begin{subfigure}{1\textwidth}
        \centering
        \includegraphics[width=\linewidth]{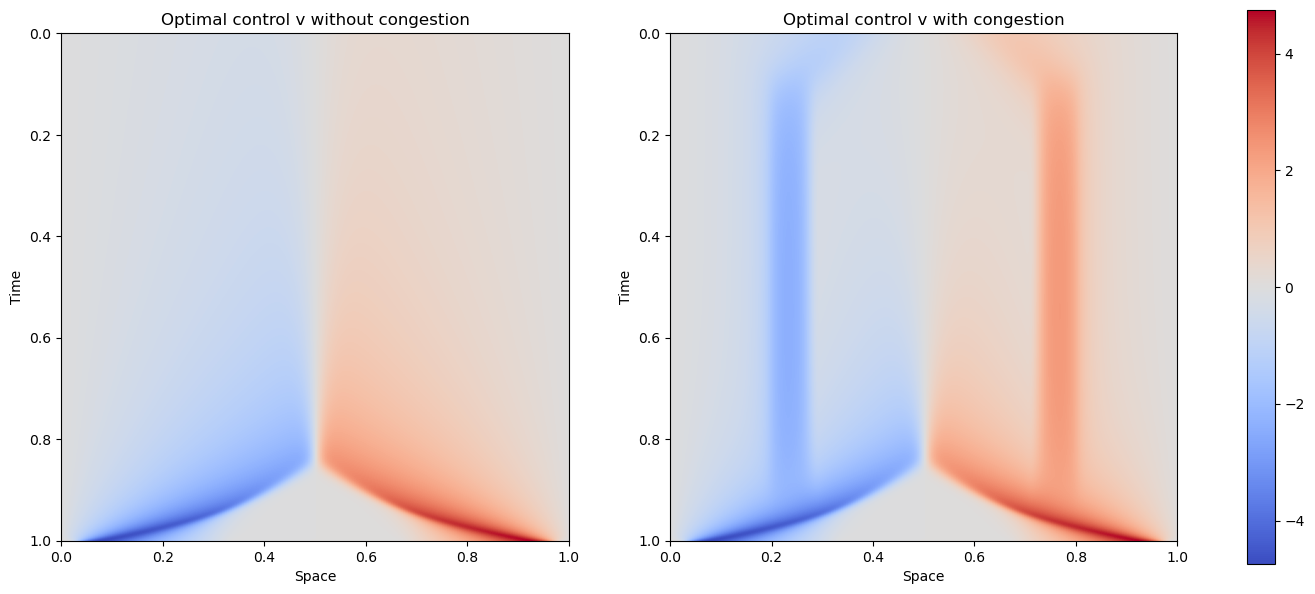}
        \caption{Comparison of optimal controls:  the case without $f^c$ (left), the case with $f^c$ (right).}
        \label{fig_control}
    \end{subfigure}
    \hfill

\vspace{0.5cm}
    
    \begin{subfigure}{1\textwidth}
        \centering
        \includegraphics[width=\linewidth]{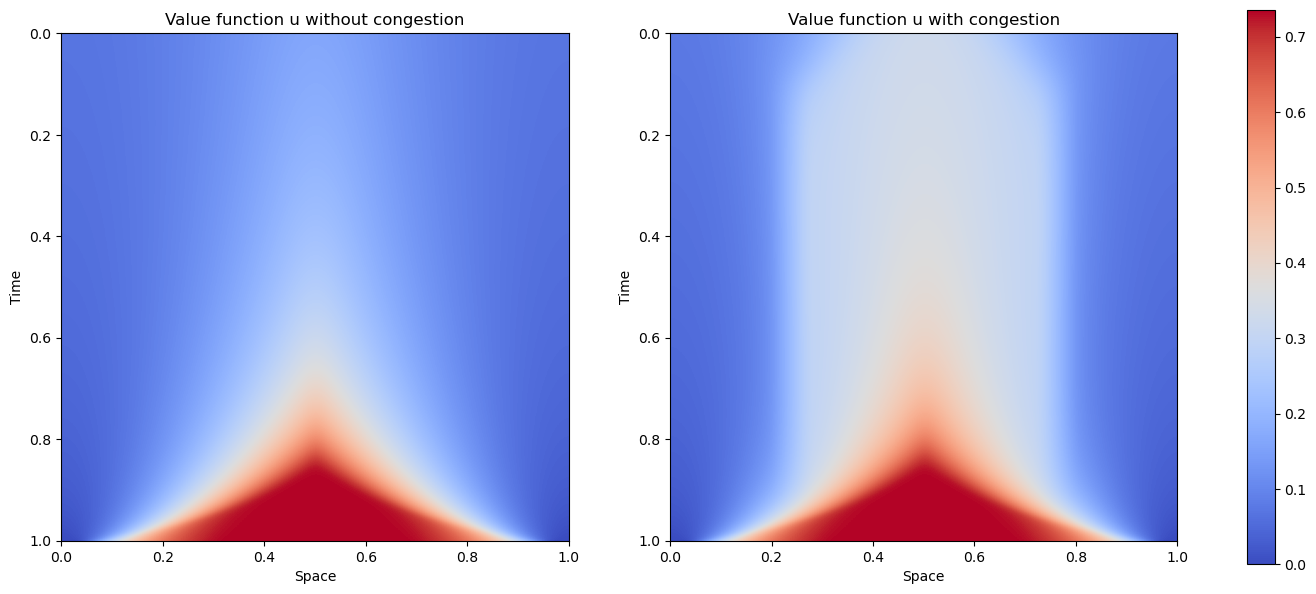}
        \caption{Comparison of value functions:  the case without $f^c$ (left), the case with $f^c$ (right).}
        \label{fig_value_function}
    \end{subfigure}
    \medskip
    \caption{Optimal controls and value functions}
    \label{fig_oc_u}
\end{figure}

We next investigate the convergence of Algorithm \ref{alg2} (for the same discretization parameters as above). We execute Algorithm \ref{alg2} with $1000$ iterations, utilizing the open-loop choice $\lambda_k = 2/(k+2)$ and the closed-loop choice \eqref{eq:lambda} (referred to as the line-search method).
We present the convergence results in Figure \ref{fig_conv}. Evaluating $\gamma_k$, equal to $\mathcal{J}(m_k,w_k) - \mathcal{J}(\bar{m},\bar{w})$ by definition, is difficult since the exact solution $(\bar{m},\bar{w})$ is not known.
On the other hand, the quantity $\bar{\gamma}_k$, which serves as an upper bound of $\gamma_k$ by  \eqref{eq:classical_bound} can directly computed in view of its definition, based on $(m_k,w_k)$ and $(\bar{m}_k,\bar{w}_k)$. Therefore, instead of evaluating $\gamma_k$, we display the evolution of $\bar{\gamma}_k$, see Figure \ref{fig_conv}. The two figures of Figure \ref{fig_conv} are the same, with different scales for the horizontal axis.
In the left part of Figure \ref{fig_conv}, we see that Algorithm \ref{alg2} exhibits a convergence rate of order $\mathcal{O}(1/k^4)$ for the choice $\lambda_k = 2/(k+2)$, which is better than the theoretical convergence rate $\mathcal{O}(1/k)$ obtained from \eqref{eq:conv1_2}. In the right part of Figure \ref{fig_conv}, a linear convergence rate can be observed for the line-search case, as predicted in \eqref{eq:conv2_2}.

Finally, we present numerical results concerning the mesh-independence of Algorithm \ref{alg2} applied to \eqref{eq:theta_mfg}. To see this, we discretize the state space with steps sizes: $h=1/250$, $h=1/500$, and $h=1/1000$. 
The corresponding step sizes for the time space are: $\Delta t = 1/500$, $\Delta t = 1/2000$, and $\Delta t = 1/8000$.
The convergence results associated with these discretization steps are displayed in Figure \ref{fig_mesh_ind}. From the left part of Figure \ref{fig_mesh_ind},  it can be observed that the convergence rate of Algorithm \ref{alg2} remains unaffected by the choice of $h$ when $\lambda_k=2/(k+2)$.
The right part of Figure \ref{fig_mesh_ind} shows that the convergence rate of Algorithm \ref{alg2} can even benefit from a refinement of the discretization parameters in the line-search case. These results are consistent with mesh-independence properties outlined in Theorems \ref{thm:main2} and \ref{thm:main3}.

\begin{figure}[htbp]
	\centering
	\includegraphics[width=1\linewidth]{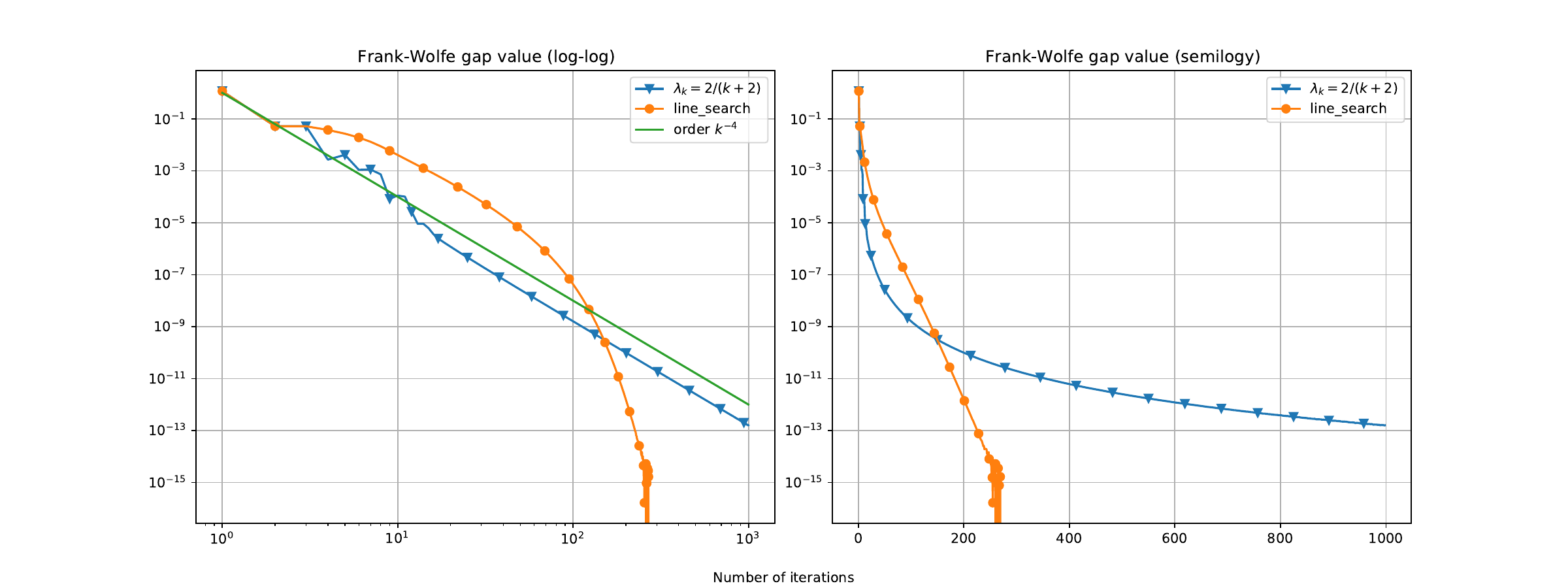}
	\caption{Convergence results of Algorithm \ref{alg2}.}
	\label{fig_conv}
\end{figure}

\begin{figure}[htbp]
	\centering
	\includegraphics[width=1\linewidth]{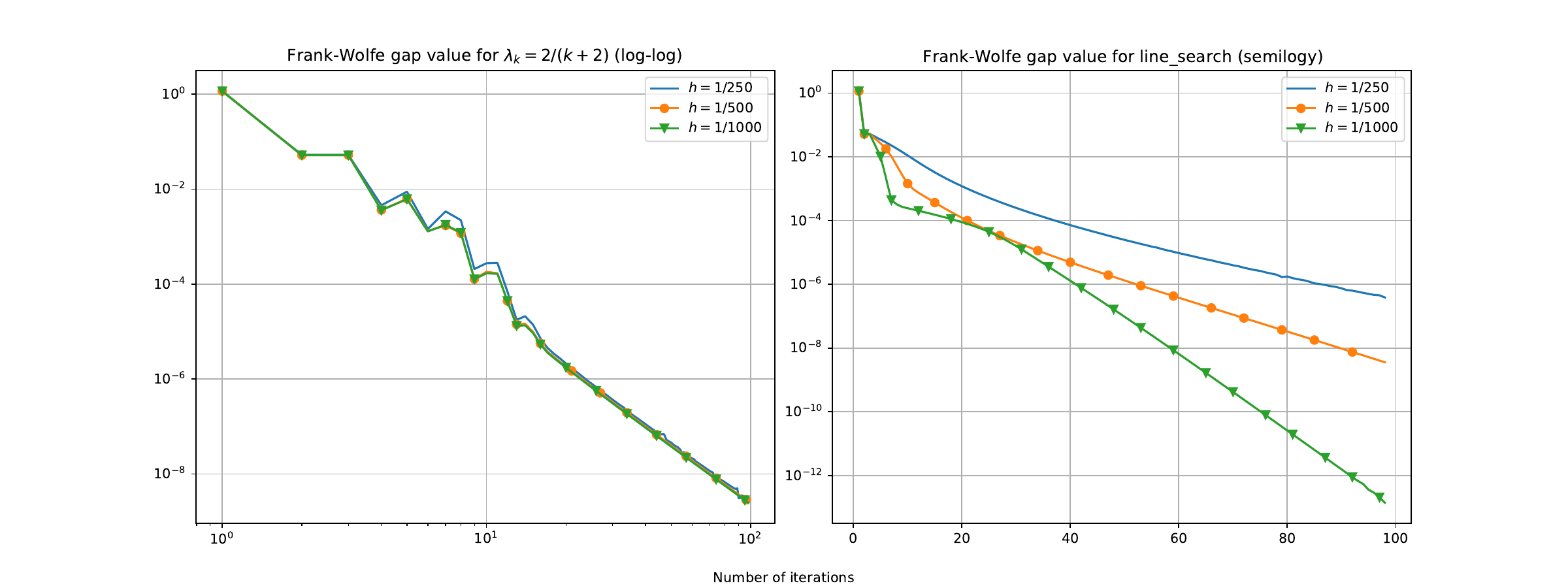}
	\caption{Mesh-independence property of Algorithm \ref{alg2}.}
	\label{fig_mesh_ind}
\end{figure}

\section{Conclusion}

We have established mesh-independent convergence results for the resolution of potential MFGs with the generalized Frank-Wolfe algorithm. This robustness property makes the GFW algorithm a method of choice for the resolution of potential MFGs. Our analysis has benefited from the intrinsic simplicity of the convergence proof of the GFW algorithm, which at the discrete level only required us to prove natural $L^2$ energy estimates for the Fokker-Planck equation (in the sublinear case) and an estimate of the semi-concavity modulus of the HJB equation (in the linear case). We expect that our analysis can be extended to the combination of the GFW algorithm with other discretization schemes, such as the implicit scheme of \cite{achdou2013mean}. Let us stress, however, that the implementation of the GFW algorithm is made difficult for such schemes, since they would imply to solve fully implicit discrete HJB equations, involving nonlinear implicit relations, while for the theta-scheme, the implicit relations are linear, thus much easier to handle.

Let us underline again that the application of the GFW algorithm requires us to interpret the MFG system as the first-order necessary optimality condition of an optimization problem, it is therefore restricted to the case of potential problems. The convexity of the potential problem is crucial in the current analysis, however, we mention that the Frank-Wolfe algorithm was extended to non-convex problems (see \cite{lacoste2016}). In a non-convex setting, the convergence to a global solution cannot be ensured, yet the convergence of some stationarity criterion can be demonstrated. Future work will aim at proving mesh-independent principles for those convergence results, in the context of non-convex mean-field-type optimal control problems.

Another line of research could focus on the extension of the current framework to the case of degenerate potential mean-field games. It would be natural to investigate the combination of the Frank-Wolfe algorithm with semi-Lagrangian schemes, which can handle models with a possibly degenerate diffusion, see \cite{gianatti2023approximation} and the references therein.

Finally, we mention the case of first-order MFGs in their Lagrangian formulation, in which the equilibrium configuration is described by a probability measure on some trajectory set. It has a prescribed marginal $m_0$, describing the initial conditions of the agents. In our preprint \cite{liu2023MFO}, we propose a general class of optimization problems, containing the potential Lagrangian MFGs. We propose a tractable variant of the Frank-Wolfe algorithm, based on our work \cite{bonnans2022large}, that is combined with a discretization of $m_0$ as an empirical distribution associated with a set of $N$ points. Our method exhibits a sublinear rate of convergence which can be qualified as mesh-independent, since it improves as $N$ increases.

\end{document}